\documentclass[10pt, reqno]{amsart}

\usepackage{amsthm,amssymb,amstext,amscd,amsfonts,amsbsy,amsrefs,amsxtra,latexsym,amsmath,xcolor,mathrsfs,fancybox,upgreek, soul,url}
\usepackage[english]{babel}
\usepackage[all,cmtip]{xy}
\usepackage[latin1]{inputenc}
\usepackage{cancel}
\usepackage[draft]{hyperref}

\usepackage{comment}
\usepackage{mdframed}
\allowdisplaybreaks
\usepackage{mathtools}
\usepackage{enumerate}
\usepackage{thmtools}
\usepackage{thm-restate}
\usepackage{chngcntr}
\usepackage{enumitem}
\usepackage{etoolbox}
\usepackage{todonotes}

\usepackage{graphicx}

\usepackage{tikz}
\usepackage{verbatim}
\usetikzlibrary{quotes,angles}

\DeclarePairedDelimiter\abs{\lvert}{\rvert}
\DeclarePairedDelimiter\norm{\lVert}{\rVert}

\makeatletter
\let\oldabs\abs
\def\abs{\@ifstar{\oldabs}{\oldabs*}}
\let\oldnorm\norm
\def\norm{\@ifstar{\oldnorm}{\oldnorm*}}
\makeatother

\makeatletter
\glb@settings 
\makeatother

\newtheorem{theorem}{Theorem}
\newtheorem{lemma}[theorem]{Lemma}
\newtheorem{corollary}[theorem]{Corollary}
\newtheorem{proposition}[theorem]{Proposition}

\theoremstyle{definition}

\theoremstyle{remark}
\newtheorem*{nonumbertheorem}{{\bf Theorem}}
\newtheorem*{remark}{Remark}

\numberwithin{theorem}{section}
\numberwithin{proposition}{section}
\numberwithin{lemma}{section}
\numberwithin{corollary}{section}
\numberwithin{equation}{section}
\numberwithin{conjecture}{section}

\setlist[enumerate,1]{before=}
\AfterEndEnvironment{enumerate}{}

\newcommand{\Arg}{\mathrm{Arg}}
\newcommand{\Log}{\operatorname{Log}}
\newcommand{\N}{\mathbb{N}}

\newcommand{\R}{\mathbb{R}}
\newcommand{\C}{\mathbb{C}}

\usepackage{bm}

\renewcommand{\pmod}[1]{\  \,  \left( \mathrm{mod} \,  #1 \right)}


\begin{document}

\title[Sign changes in statistics for plane partitions]{Sign changes in statistics for plane partitions}
\author{Walter Bridges, Johann Franke, and Joshua Males}

\begin{abstract}
	Recent work of Cesana, Craig and the third author shows that the trace of plane partitions is asymptotically equidistributed in residue classes mod $b$.  Applying a technique of the first two authors and Garnowski, we prove asymptotic formulas for the secondary terms in this equidistribution, which are controlled by certain complex numbers generated by a twisted MacMahon-type product.  We further carry out a similar analysis for a statistic related to plane overpartitions.
\end{abstract}

\maketitle

\section{Introduction and statement of results}

Partitions are ubiquitous in number theory, and in the last century have become one of the most well-studied combinatorial objects. For example, if $p(n)$ denotes the number of partitions of a positive integer $n$, Hardy and Ramanujan \cite{HardyRama} introduced their famed Circle Method in order to prove that as $n \to \infty$ we have
\begin{align*}
p(n) \sim \frac{1}{4 n \sqrt{3}} \exp\left(\pi \sqrt{\frac{2n}{3}}\right),
\end{align*}
a technique which was perfected by Rademacher \cite{Rademacher}, with further versions introduced by Wright \cites{Wright1,Wright2} among many others. Partitions also appear in many areas outside of number theory and combinatorics, such as in algebraic topology \cite{BCMO} and representation theory \cites{FS,GO}.

One of the most natural generalizations of partitions are the so-called \textit{plane partitions}. A plane partition of $n \in \mathbb{N}$ (see e.g.\@ \cite{Andrews}) is a two-dimensional array $\pi=\{\pi_{j,k}\}_{j, k \geq 1}$ of non-negative integers that is non-increasing in both indices, i.e $\pi_{j,k} \geq \pi_{j+1,k}$, $\pi_{j,k} \geq \pi_{j,k+1}$ for all $j$ and $k$, and fulfills  $|\pi| := \sum_{j,k} \pi_{j,k} = n$. Plane partitions can be visualized in the plane as in Figure 1.

\begin{figure}[h]
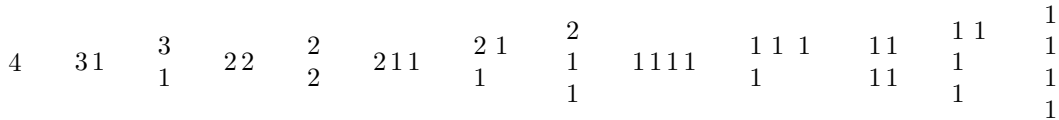
 \label{F:PPexample}
    \begin{center}
    $$
    4 \qquad \begin{matrix} 3 & \hspace{-.3cm} 1 \end{matrix} \qquad \begin{matrix} 3 \\ 1 \end{matrix} \qquad \begin{matrix} 2 & \hspace{-.3cm} 2 \end{matrix} \qquad \begin{matrix} 2 \\ 2 \end{matrix} \qquad \begin{matrix} 2 & \hspace{-.3cm} 1 & \hspace{-.3cm} 1 \end{matrix} \qquad \begin{matrix} 2 & \hspace{-.3cm} 1 \\ 1 & \end{matrix} \qquad \begin{matrix} 2 \\ 1 \\ 1  \end{matrix} \qquad \begin{matrix} 1 & \hspace{-.3cm} 1 & \hspace{-.3cm} 1 & \hspace{-.3cm} 1  \end{matrix} \qquad \begin{matrix} 1 & \hspace{-.3cm} 1 & \hspace{-.3cm} 1 \\ 1  & & \end{matrix} \qquad \begin{matrix} 1 & \hspace{-.3cm} 1  \\ 1  & \hspace{-.3cm} 1  \end{matrix}
    \qquad
    \begin{matrix} 1 & \hspace{-.3cm} 1  \\ 1  & \\ 1 & \end{matrix} \qquad \begin{matrix} 1 \\ 1  \\ 1   \\ 1  \end{matrix}
    $$

   \caption{ All plane partitions of size 4.}
   \end{center} 
     
\end{figure}

 Within number theory plane partitions have many applications, with an excellent overview given in the classical reference \cite{Stanley2}. They also appear in areas outside of number theory and combinatorics, and have seen a surge in interest in recent years. For example, they are intricately related to certain aspects of superconformal field theory \cite{Ok} and in connection to counting small black holes in string theory \cite{DDMP}. In the present paper, we determine refined information on certain statistics for plane partitions, complementing the growing literature in the area.

Let $pp(n)$ count the number of plane partitions of $n$ (with $pp(0):=1$).  A classical result of MacMahon \cite{MacMahon} gives the product generating function
\begin{equation}\label{E:PPgeneratingfunction}\sum_{\pi} q^{|\pi|}=\sum_{n\geq 0}pp(n) q^n = \prod\limits_{n\geq 1} \dfrac{1}{(1 - q^n)^n}.\end{equation}
One well-studied statistic associated a plane partition $\pi$ is its {\it trace} $t(\pi)$, which is the sum of the diagonal entries
$$t(\pi) = \sum_{j\geq 1} \pi_{j,j}.$$  Let $pp(m,n)$ denote the number of plane partitions with trace $m$.  Stanley \cite{Stanley} applied a bijection of Bender and Knuth \cite{BenderKnuth} to show that the trace is generated by a simple refinement of \eqref{E:PPgeneratingfunction},
\begin{equation}\label{E:twovargenfn}PP(\zeta;q) :=\sum_{\pi} \zeta^{t(\pi)}q^{|\pi|}= \sum_{n,m\geq 0} pp(m,n) \zeta^m q^n = \prod\limits_{n\geq 1} \dfrac{1}{(1 - \zeta q^n)^n}.\end{equation}
The coefficient of $q^n$ above is a polynomial in $\zeta$, which we denote by $T_n(\zeta)$.

To enumerate plane partitions with trace in residue classes, we can substitute roots of unity for $\zeta$ and use orthogonality.  Let $pp(a,b,n)$ denote those plane partitions of $n$ with trace congruent to $a\pmod{b}$ and let $\zeta_k^h:=e^{2\pi i \frac{h}{k}}$.  Then
\begin{equation}\label{E:ppintermsofT}
pp(a,b,n) =\frac{pp(n)}{b}+\frac{1}{b}\sum_{1 \leq \nu \leq b-1} \zeta_b^{- a\nu }T_n(\zeta_b^{ \nu }).
\end{equation}
Work of Cesana, Craig and the third author shows that the term $pp(n)$ dominates the right-hand side as $n\to \infty$; that is, we have asymptotic equidistribution of the trace in residue classes mod $b$.
\begin{nonumbertheorem}[Theorem 1.4, \cite{CCM}]
For all $a,b\in \mathbb{N}$, we have, as $n \to \infty$,
$$
\frac{pp(a,b,n)}{pp(n)} \sim \frac{1}{b}.
$$
\end{nonumbertheorem}
Secondary terms in the asymptotic behavior of $pp(a,b,n)$ are controlled by the polynomials $T_n(\zeta_b^{\nu})$.  In particular, consider two distinct residue classes $a_1 \not\equiv a_2 \pmod{b}$.  Then in the difference \newline $pp(a_1,b,n)-pp(a_2,b,n)$, the term $pp(n)$ in \eqref{E:ppintermsofT} cancels and to predict the oscillation in this difference as $n \to \infty$, we require the asymptotic behavior of the complex numbers $T_n(\zeta_b^{\nu})$.  

To state our main result, recall the trilogarithm which is defined for $|z| \leq 1$ by the series $\mathrm{Li}_3(z)=\sum_{n \geq 1} \frac{z^n}{n^3}.$  Let $\theta_{12}=0.47585\dots$ be the solution to
$$
\mathrm{Re} \left(\sqrt[3]{\mathrm{Li}_3\left(e^{2\pi i \theta}\right)}\right) =\frac{\mathrm{Re}\left(\sqrt[3]{\mathrm{Li}_3\left(e^{4\pi i \theta}\right)}\right)}{2}, \qquad 0 \leq \theta \leq \frac{1}{2}.
$$
(See Corollary \ref{C:TrilogDominant}.)  Then we have the following asymptotic formulas.  As $T_n\left(\overline{\zeta}\right)=\overline{T_n(\zeta)},$ we state them only for Im$(\zeta) \geq 0$.

\begin{theorem}\label{T:PPasymp}  Let $a,b$ be integers with $1 \leq a \leq b$ and $\gcd(a,b)=1$. 
\begin{enumerate}
    \item If $0< \frac{a}{b}< \theta_{12}$, then $$  T_n\left(\zeta_b^a\right) \sim \frac{\left(1-\zeta_b^a\right)^{\frac{1}{12}}\mathrm{Li}_3\left(\zeta_b^a\right)^{\frac{1}{6}}}{2^{\frac{1}{3}}\sqrt{3\pi}n^{\frac{2}{3}}}\exp\left(\frac{3}{2^{\frac{2}{3}}}\mathrm{Li}_3\left(\zeta_b^a\right)^{\frac{1}{3}}n^{\frac{2}{3}}\right).$$
    \item If $\theta_{12} < \frac{a}{b} < \frac{1}{2}$, then 
    $$  T_n\left(\zeta_b^a\right) \sim (-1)^n \frac{\left(1-\zeta_b^a\right)^{\frac{1}{6}} \mathrm{Li}_3\left(\zeta_b^{2a}\right)^{\frac{1}{6}}}{\left(1+\zeta_b^a\right)^{\frac{1}{12}}2^{\frac{5}{6}}\sqrt{3\pi}n^{\frac{2}{3}}}\exp\left(\frac{3}{2^{\frac{5}{3}}}\mathrm{Li}_3\left(\zeta_b^{2a}\right)^{\frac{1}{3}}n^{\frac{2}{3}}\right).$$
    \item We have
    $$
    T_n(-1) \sim (-1)^n \frac{e^{-\zeta'(-1)} \zeta(3)^{\frac{5}{36}}}{2^{\frac{3}{4}}\sqrt{3\pi}n^{\frac{23}{36}}}\mathrm{exp}\left(\frac{3}{2^{\frac{5}{3}}}\zeta(3)^{\frac{1}{3}}n^{\frac{2}{3}} \right).
    $$
\end{enumerate}
\end{theorem}

\begin{remark}
The case $\zeta=1$ is the classical asymptotic formula for plane partitions due to Wright \cite{Wright3},
$$
    T_n(1)=pp(n) \sim \frac{\zeta(3)^\frac{7}{36}e^{\zeta'(-1)}}{2^{\frac{11}{36}}\sqrt{3\pi}n^\frac{25}{36}} \exp\left(\frac{3\zeta(3)^\frac13}{2^\frac23}n^\frac23\right).
$$
When $\zeta=-1$, we have $T_n(-1)=pp(0,2,n)-pp(1,2,n),$ the difference of plane partitions with even and odd trace.
\end{remark}

From our asymptotic formulas, it follows that the differences $pp(a_1,b,n)-pp(a_2,b,n)$, when rescaled, oscillate like a cosine.  (See Figure 2.)

\begin{figure}[h]
         \centering
         \includegraphics[width=90mm,height= 50mm]{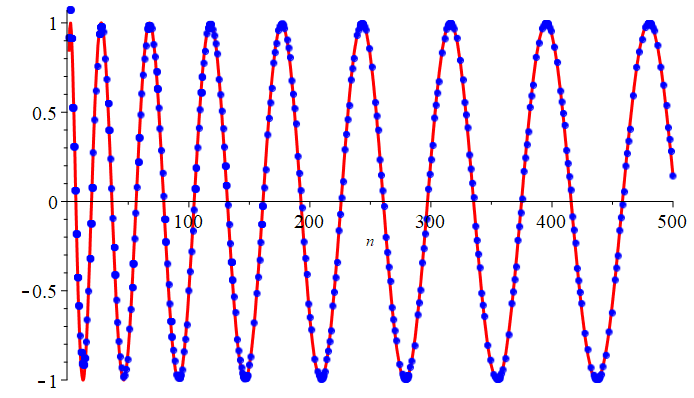}
        
\label{fig:difffigintro}

\caption{The blue dots depict the function $\frac{pp(1,5,n)-pp(4,5,n)}{Bn^{-\frac34}e^{3 \cdot 2^{-\frac{2}{3}}\lambda_1 n^{\frac{2}{3}}}}$, and the red line is the asymptotic prediction  $\cos\left(\alpha + 3 \cdot 2^{-\frac{2}{3}}\lambda_2 n^{\frac{2}{3}}\right)$. The approximate values of the constants are $B \approx 0.19971$, $\alpha \approx -1.41897$, $\lambda_1 \approx 0.89873$, and $\lambda_2\approx 0.44610$.}
\end{figure}

\begin{corollary} \label{C:pp-signchanges} Let $b \geq 3$, and $a_1 \not\equiv a_2$ be two classes modulo $b$. Then we have 
\begin{align*}
    \frac{pp(a_1, b, n) - pp(a_2, b, n)}{B_{a_1, a_2, b} n^{-\frac23} \exp\left(\frac{3}{2^{\frac23}} \lambda_1 n^{\frac{2}{3}}\right)} = \cos\left( \alpha_{a_1, a_2, b} + \frac{3}{2^{\frac{2}{3}}} \lambda_2 n^{\frac{2}{3}}\right) + o(1),
\end{align*}
where $\lambda_1 + i\lambda_2 := \mathrm{Li}_3(\zeta_b)^{\frac13}$, and $B_{a_1, a_2, b} > 0$ and $\alpha_{a_1, a_2, b} \in [0, 2\pi)$ are defined by 
\begin{align*}
    B_{a_1, a_2, b} e^{i\alpha_{a_1, a_2, b}} = \frac{2^{\frac{2}{3}}}{b\sqrt{3\pi}} \left( \zeta_{b}^{-a_1} - \zeta_b^{-a_2}\right) \left( 1 - \zeta_b\right)^{\frac{1}{12}} (\lambda_1 + i\lambda_2)^{\frac{1}{2}}.
\end{align*}
\end{corollary}

\begin{remark}
Corollary \ref{C:pp-signchanges} can also be generalized to the case of higher differences, i.e. terms $\sum_{a \pmod{b}} v_a pp(a,b,n)$ with $\sum_{a \pmod{b}} v_a = 0$, completely analogously to the case of ordinary partitions in \cite{BFG}.  
\end{remark}

We prove Theorem \ref{T:PPasymp} using an adaptation of the Hardy--Ramanujan Circle Method due to the first two authors and Garnowski \cite{BFG}.  Asymptotics for $T_n(\zeta)$ in most cases with $|\zeta| \neq 1$ were proved by Boyer and Parry \cite{BoyerParry1,BoyerParry2,BoyerParry3} in connection with the zero attractors of these polynomials.  The case of generic products was considered by Parry in \cite{Parry}, again for $|\zeta|\neq 1.$  As in \cite{BFG}, the primary difficulty is the bounding of minor arcs, as classical techniques do not apply when $\zeta$ is complex and has absolute value 1.  The present work demonstrates the wide applicability of the techniques in \cite{BFG} to twisted infinite product generating functions.

We also consider the following MacMahon-type product,
$$
\overline{PP}(\zeta;q)=\sum_{n \geq 0} \mathcal{A}_n(\zeta)q^n:=\prod_{n \geq 1} \left(\frac{1-\zeta q^n}{1-q^n}\right)^n.
$$
For an explicit description of the statistics for plane partitions generated by $\mathcal{A}_n(\zeta)$, we refer the readers to the work of Vuleti\'{c} \cite{Vul1,Vul2}.  If one substitutes $\zeta=-1$, then $\overline{PP}(-1;q)$ counts {\it plane overpartitions} (see \cite{CSV}).  Again, the work of Cesana, Craig and the third author \cite{CCM} implies asymptotic equidistribution in residue classes, and we extend this to secondary terms by proving the following asymptotic formulas.

\begin{theorem}\label{T:planeoverpartitions}
For $\frac{a}{b} \neq 1$, we have
$$
\mathcal{A}_n\left(\zeta_b^a\right) \sim \left(1-\zeta_b^a\right)^{-\frac{1}{12}}\frac{e^{\zeta'(-1)}\left(\zeta(3)-\mathrm{Li}_3\left(\zeta_b^a\right)\right)^{\frac{7}{36}}}{2^{\frac{11}{36}}\sqrt{3\pi}n^{\frac{25}{36}}}\exp\left(\frac{3}{2^{\frac{2}{3}}}\left(\zeta(3)-\mathrm{Li}_3\left(\zeta_b^a\right)\right)^{\frac{1}{3}}  n^{\frac{2}{3}}\right).
$$
\end{theorem}

In particular, we have the following asymptotic formula for the number of plane overpartitions of size $n$, $\overline{pp}(n):=\mathcal{A}_n(-1)$.

\begin{corollary}
As $n \to \infty,$ we have
$$
\overline{pp}(n) \sim  \frac{e^{\zeta'(-1)}(7\zeta(3))^{\frac{7}{36}}}{2^{\frac{7}{9}}\sqrt{3\pi}n^{\frac{25}{36}}} \mathrm{exp}\left(\frac{3}{2^{\frac{4}{3}}}(7\zeta(3))^{\frac{1}{3}}n^{\frac{2}{3}}\right).
$$
\end{corollary}

We structure the article as follows.  Section \ref{S:prelims} recalls a few classical results useful for asymptotic analysis.  In Section \ref{S:trilog}, we make a careful study of the trilogarithm on the unit circle; this is needed to identify the major arcs in the circle method.  In Section \ref{S:proofmain}, we use the techniques of the first two authors and Garnowski \cite{BFG} to prove Theorems \ref{T:PPasymp} and \ref{T:planeoverpartitions}.  Concluding remarks regarding further extensions to generic product generating functions are given in Section \ref{S:Outlook}.

\section*{Acknowledgements}  We are grateful to Daniel Parry for his commentary on the proof of Lemma \ref{L:argzeta3trilog}.  We thank Giulia Cesana, William Craig and Taylor Garnowski for helpful feedback that greatly improved the exposition.

The first author is partially supported by the SFB/TRR 191 ``Symplectic Structures in Geometry, Algebra and Dynamics'', funded by the DFG (Projektnummer 281071066 TRR 191). The second author is partially supported by the Alfried Krupp prize.  The research of the third author conducted for this paper is supported by the Pacific Institute for the Mathematical Sciences (PIMS). The research and findings may not reflect those of the Institute

\section{Preliminaries}\label{S:prelims}

\subsection{Major arcs}

This section contains results we use in the course of evaluating the major arcs in the circle method.  We first recall Laplace's method.

\begin{theorem}[Section 1.1.5 of \cite{Pinsky}]\label{T:laplacemethod}
   Let $A,B: [a,b]\to \mathbb{C}$ be continuous functions. Suppose $x\neq x_0 \in [a,b]$ such that $\mathrm{Re}(B(x))<\mathrm{Re}(B(x_0)),$ and that 
   \begin{align*}
       \lim_{x\to x_0}\frac{B(x)-B(x_0)}{(x-x_0)^2} = -k\in \mathbb{C},
   \end{align*}
   with $\mathrm{Re}(k)>0.$ Then as $t \to \infty$
   \begin{align*}
       \int^b_{a}A(x)e^{tB(x)}dx = e^{tB(x_0)}\left(A(x_0)\sqrt{\frac{\pi}{tk}}+o\left(\frac{1}{\sqrt{t}}\right)\right).
   \end{align*}
\end{theorem}

We recall classical Euler--Maclaurin summation.

\begin{theorem} [p.\@ 66 of \cite{IwanKow}]\label{T:EulerMacclassic}
    Let $\{x\}:=x - \lfloor x \rfloor$ denote the fractional part of $x$. For $N \in \mathbb{N}$ and \newline $f:[1, \infty) \to \mathbb{C}$ a continuously differentiable function, we have
    $$
    \sum_{1 \leq n \leq N} f(n)=\int_1^N f(x)dx+\frac{1}{2}(f(N)+f(1))+\int_1^N f'(x)\left(\{x\}-\frac{1}{2}\right)dx.
    $$
\end{theorem}

We also need the following lemma that may be proved with elementary calculus.
\begin{lemma}\label{L:cosinemax}
Let $a \in \left(-\frac{\pi}{2},\frac{\pi}{2}\right)$.  Then $x=\frac{a}{3}$ is the unique maximum of the function $\cos(a-2x)\cos^2(x)$ for $x\in \left(-\frac{\pi}{2},\frac{\pi}{2}\right)$. 
\end{lemma}

\subsection{Minor arcs}

This section contains lemmas we use to bound the minor arcs in the circle method.  The following is known as Abel-partial summation.
  \begin{proposition}[p.\@ 3 of \cite{Tenenbaum}]\label{P:Abelpartialsummation}
Let $N \in \N_0$ and $M \in \N$. For sequences $\{a_n\}_{n \geq N}$, $\{b_n\}_{n \geq N}$ of complex numbers, if $A_n:=\sum_{N < m \leq n} a_m,$ 
  $$
  \sum_{N < n \leq N+M} a_nb_n=A_{N+M}b_{N+M}+\sum_{N < n < N+M} A_n(b_n-b_{n+1}).
  $$
  \end{proposition}
  
  We use a version of Euler-Maclaurin summation that applies in cones when the summand has a simple pole at $0$.  Recall that the Bernoulli polynomials $B_n(x)$ may be defined through the power series
  $$
  \sum_{n \geq 0} \frac{B_n(x)t^n}{n!}=\frac{te^{tx}}{e^t-1}.
  $$
\begin{theorem}[Theorem 1.3 of \cite{Eulermacbringmann}]\label{T:EulerMac}
Let $D_\vartheta:= \{re^{i\phi}: r\geq 0\; \textnormal{and}\; |\phi|\leq \vartheta\}$ and suppose that $0\leq \vartheta< \frac{\pi}{2}$. Let $f:\mathbb{C}\to \mathbb{C}$ be holomorphic in a domain containing $D_\vartheta$ with the exception of a simple pole at the origin, and assume that $f$ and all of its derivatives are of sufficient decay in $D_\vartheta$. If $f(w) = \sum_{n\geq -1}b_nw^n$ near $0$, then for $a\in \mathbb{R}\setminus \left(-\mathbb{N}_0\right)$, and $N\in \mathbb{N}_0$, uniformly as $w\to 0$ in $D_\vartheta$,
\begin{align*}
    \sum_{m\geq 0}f(w(m+a))=& \frac{b_{-1}\Log(w)}{w}+\frac{b_{-1}C_a}{w}+\frac{1}{w}\int^{\infty}_{0}\left(f(x)-\frac{b_{-1}e^{-x}}{x}\right)\\
    &-\sum^{N-1}_{n=0}\frac{B_{n+1}(a)b_n}{n+1}w^n+O_N\left(w^N\right),
\end{align*}
where $$ C_a:= (1-a)\sum_{m\geq 0} \frac{1}{(m+a)(m+1)}.$$
\end{theorem}
  
  The following bound for holomorphic functions is a consequence of the mean value theorem.
  
  \begin{lemma} \label{L:DifferenceEstimate} Let $f : U \to \C$ be a holomorphic function and $\overline{B_r(c)} \subset U$ a compact disk. Then, for all $a,b \in B_r(c)$ with $a \not= b$, we have 
\begin{align*}
    \left| f(b) - f(a)\right| \leq \sup_{|z - c|=r} |f'(z)| |b-a|.
\end{align*}

\end{lemma}

We apply this lemma to the function
\begin{align} \label{D:phia}
\phi_a(w):=\frac{e^{-w-aw}}{(1-e^{-w})^2} + \frac{ae^{-aw}}{1-e^{-w}} - \frac{1}{w^2},
\end{align}
where $0 \leq a \leq 1$ is a real parameter. Note that $\phi_a$ is holomorphic at 0 and in the cone $|\Arg(w)| \leq \frac{\pi}{2}- \eta,$ for any $\eta>0.$

\begin{lemma} \label{SmallRadiusEstimate} Let $x$ be a complex number with positive imaginary part and $|x| \leq 1$. Then there is a constant $c > 0$ independent from $a$ and $x$, such that for all $m \leq \frac{1}{|x|}$ we have
\begin{align*}
\left| \phi_a(xm) - \phi_a(x(m+1))\right| \leq c|x|.
\end{align*}
\end{lemma}
\begin{proof} One uses Lemma \ref{L:DifferenceEstimate} and the maximum modulus principle. The details are analogous to Lemma 2.12 in \cite{BFG}.
\end{proof}

\begin{lemma} \label{LargeRadiusEstimate} Let $x$ be a complex number with positive imaginary part. Then all $m > \frac{1}{|x|}$ we have
\begin{align*}
\left| \phi_a(xm) - \phi_a(x(m+1))\right| \ll \frac{1}{|x|^2} \left( \frac{1}{m^2} - \frac{1}{(m+1)^2}\right) + |x|e^{-m\mathrm{Re}(x)} + |x| e^{-m\mathrm{Re}(x)(1+a)} + a^2|x| e^{-am\mathrm{Re}(x)}.
\end{align*}
\end{lemma}
\begin{proof} 
We estimate the different parts separately. First we have 
\begin{align*}
    \left| \frac{1}{x^2m^2} - \frac{1}{x^2(m+1)^2}\right|  = \frac{1}{|x|^2} \left( \frac{1}{m^2} - \frac{1}{(m+1)^2}\right).
\end{align*}
For the second term in \eqref{D:phia} we obtain 
\begin{align*}
    & \left| \frac{ae^{-amx}}{1-e^{-mx}} - \frac{ae^{-a(m+1)x}}{1-e^{-(m+1)x}} \right|  \leq  ae^{-am\mathrm{Re}(x)} \left( \left| \frac{1}{1-e^{-mx}} - \frac{1}{1-e^{-(m+1)x}} \right| + \left| \frac{e^{-ax}-1}{1-e^{-(m+1)x}} \right| \right)
\end{align*}
and since $m > \frac{1}{|x|}$, we obtain
\begin{align*}
    \ll a|x|e^{-m\mathrm{Re}(x)} + a^2|x|e^{-a\mathrm{Re}(x)m}.
\end{align*}
For the first term, we obtain 
\begin{align*}
     \left| \frac{e^{-mx(a+1)}}{(1 - e^{-mx})^2} - \frac{e^{-(m+1)x(a+1)}}{(1 - e^{-(m+1)x})^2} \right|  &= e^{-m\mathrm{Re}(x)(a+1)} \left| \frac{1}{(1 - e^{-mx})^2} - \frac{1}{(1 - e^{-(m+1)x})^2} + O(x)\right| \\
    & \ll |x| e^{-m\mathrm{Re}(x)(a+1)},
\end{align*}
where we have used $|xm| > 1$ and the mean value theorem. 
\end{proof}

Finally, we need some bounds from \cite{BFG} on partial sums of the twisted harmonic series,
\begin{align*}
G_M(\theta) := \sum_{1 \leq m \leq M} \frac{e^{2\pi i\theta m}}{m}.
\end{align*}

\begin{lemma}[Lemma 2.16 of \cite{BFG}] \label{L:CosSumBound} We have, uniformly for $0 < \theta < 1$, 
	\begin{align*}
	\left| G_M(\theta) \right| \ll \log\left(\frac{1}{\theta}\right) + \log\left( \frac{1}{1-\theta}\right), \qquad M \to \infty. 
	\end{align*}
	
	\end{lemma}
	
	\begin{lemma}[Lemma 2.17 of \cite{BFG}] \label{L:Gmaxbound} Let $a$, $b$, $h$ and $k$ be positive integers, such that $\gcd(a,b) = \gcd(h,k) = 1$. We assume that $a$ and $b$ are fixed. Then we have, as $k \to \infty$ 
\begin{align*}
    \sum_{\substack{1 \leq j \leq k \\ ak+bjh \not\equiv 0 \pmod{bk}}} \max_{m \geq 1} \left| G_M\left( \frac{a}{b} + \frac{hj}{k}\right) \right| = O(k)
\end{align*}
uniformly in $h$. 
\end{lemma}

\section{Trilogarithms on the unit circle}\label{S:trilog}

We recall that for complex numbers $s$ and $z$ with $|z| < 1$ the polylogarithm $\mathrm{Li}_s(z)$ is defined by the series
\begin{align*}
    \mathrm{Li}_s(z) := \sum_{n \geq 1} \frac{z^n}{n^s}.
\end{align*}
We are especially interested in the case $s=3$, where $\mathrm{Li}_3(z)$ is called the trilogarithm. In this section, we derive some tools concerning the trilogarithm that are crucial for the choice of major arcs in the circle method. In doing so, we follow ideas of Boyer and Parry, who have already made similar considerations in detail for the dilogarithm $\mathrm{Li}_2(z)$ in the appendix \textit{Analysis of the Root Dilogarithm} of \cite{BoyerParry4}. Some of the following results regarding the trilogarithm have already been proved by Boyer and Parry \cite{BoyerParry3}. To determine major arcs in the circle method for $T_n(\zeta)$, we need to study the function
\begin{align} \label{eq:importantfunction}
\theta \mapsto \mathrm{Re}\left(\sqrt[3]{\mathrm{Li}_3(e^{i \theta})}\right),
\end{align}
for $\theta \in (0,\pi).$ Here we are primarily interested in monotonicity questions, since these will later answer which parts of our asymptotic formulas dominate others. 
In general, the mapping behavior of power series is not an easy problem. However, under assumptions on the coefficients, results can be obtained in this direction. This concerns, for instance, questions whether holomorphic mappings map the unit circle onto convex sets, or monotonicity problems. One possibility is to consider higher differences of the coefficients. This plays into our hands above all in the theory of polylogarithms, where the coefficients behave very uniformly. For a sequence of real numbers $a := a_n$, we define 
\begin{align*}
    \Delta^r(a)_n := \sum_{j=0}^r \binom{r}{j} (-1)^j a_{n+j}.
\end{align*}
We call $a_n$ strictly $r$-fold monotone, if $\Delta^m(a)_n > 0$ for all $n \in \N$ and $1 \leq m \leq r$. We need the following theorem by Fejer. 

\begin{theorem}[see \cite{Fejer}] \label{T:Fejer} Let $(a_n)_{n \in \N}$ be a sequence of real numbers that is strictly 4-fold monotone and satisfies $\lim_{n \to \infty}a_n = 0$. Then the function $F$, that is defined by the Fourier series 
\begin{align*}
    F(\theta) := \sum_{n=1}^\infty a_n \cos(n\theta)
\end{align*}
on the interval $(0, 2\pi)$, is strictly decreasing on the interval $(0, \pi)$.\footnote{In fact, Fejer shows something slightly different: He only assumes that the generating sequence is 4-fold monotone in the sense $\Delta^m(a) \geq 0$ and concludes that the function $\theta \mapsto P(e^{i\theta})$ is monotone in $(0, \pi)$. However, as he does exclude the case $a_j=0$ for all $j$, he argues that there is a $\mu$ such that $\Delta^4(a)_\mu > 0$. With this, one can actually show that the derivative is negative and the function is strictly decreasing.} 
\end{theorem}
The proof uses multiple Abel-partial summation. We can use this result to obtain the following statement about the boundary behavior of power series. 
\begin{theorem} \label{T:Fejer2} Let $(a_n)_{n \in \N}$ be a sequence of real numbers that is strictly 4-fold monotone. Assume that the series
\begin{align*}
    P(z) := \sum_{n=1}^\infty a_n z^n
\end{align*}
converges absolutely on $|z| \leq 1$. Then the function $t \to |P(e^{it})|$ is strictly decreasing on the interval $(0, \pi)$.
\end{theorem}
\begin{proof} We follow Fejer \cite{Fejer2}. First, one has
\begin{align*}
    \left|P\left(e^{it}\right)\right|^2 = \sum_{n=1}^\infty A_n \cos(nt)
\end{align*}
where
\begin{align*}
    A_n := \sum_{j=1}^\infty a_j a_{n+j}.
\end{align*}
As the sequence $(a_n)_{n \in \N}$ is strictly 4-fold monotone, so is $(A_n)_{n \in \N}$. The theorem now follows from Theorem \ref{T:Fejer}. 
\end{proof}
The concatenation of real part and root in \eqref{eq:importantfunction} is not easy to handle in monotonicity questions at first. Therefore we use a trick and split the problem into argument and absolute value. The idea of proof of the following Lemma is adapted from Boyer and Parry \cite{BoyerParry4} (Appendix A, Lemma 7). 
\begin{lemma} \label{L:Lidecreasing} Let $k \geq 2$. Then the function $\theta \mapsto |\mathrm{Li}_k(e^{i\theta})|$ is decreasing on $(0, \pi)$. 
\end{lemma}
\begin{proof} By Theorem \ref{T:Fejer2}, it is sufficient to show that the sequence $a_n := \frac{1}{n^k}$ is strictly 4-fold monotone. To verify this, use the identity 
\begin{align*}
    \frac{1}{n^k} = \frac{1}{\Gamma(k)} \int_0^1 x^{n-1} (-\log(x))^{k-1} dx
\end{align*}
to find that for $n \geq 1$,
\begin{align*}
    \Delta^m\left(\frac{1}{n^k}\right) & = \sum_{j=0}^m {m \choose j} \frac{(-1)^{j}}{(n+j)^k} = \frac{1}{\Gamma(k)} \int_0^1 \sum_{j=0}^m {m \choose j} (-1)^j x^{n+j-1} (-\log(x))^{k-1} dx \\
    & = \frac{1}{\Gamma(k)} \int_0^1 x^{n-1} \sum_{j=0}^m {m \choose j} (-1)^j x^{j} (-\log(x))^{k-1} dx = \frac{1}{\Gamma(k)} \int_0^1 x^{n-1}(1-x)^m (-\log(x))^{k-1} dx > 0.
\end{align*}

\end{proof}
For the convenience of the reader we provide detailed proofs for Proposition \ref{P:trilogdecrease} and Corollary \ref{C:TrilogDominant}. Note that these results have been proved by Boyer and Parry \cite{BoyerParry3}. 
\begin{proposition}\label{P:trilogdecrease}
The function $\theta \mapsto \mathrm{Re}\left(\mathrm{Li}_3(e^{2\pi i \theta})^{\frac{1}{3}}\right)$ is decreasing on $\left(0,\frac{1}{2}\right).$
\end{proposition}
\begin{proof} First we argue that the function $t \mapsto \mathrm{Arg}(\mathrm{Li}_3(e^{it}))$ is increasing on the interval $(0, \pi)$. By a result of Lewis \cite{Lewis}, any polylogarithm $\mathrm{Li}_3(z)$ is star-like and univalent on the unit disk. Now, by Theorem 2.10 of \cite{Duren} on p. 41, this is equivalent to 
\begin{align*}
    \frac{d}{dt} \left( \mathrm{Arg}(\mathrm{Li}_3(\rho e^{it}))\right) = \mathrm{Re}\left({\frac{\rho e^{it} \mathrm{Li}_3'(\rho e^{it})}{\mathrm{Li}_3(\rho e^{it})}}\right) > 0, \qquad 0 < \rho < 1.
\end{align*}
By taking limits, we conclude 
\begin{align*}
    \mathrm{Re}\left({\frac{e^{it} \mathrm{Li}_3'(e^{it})}{\mathrm{Li}_3(e^{it})}}\right) \geq 0.
\end{align*}
By Lemma \ref{L:Lidecreasing}, we know that the function $t \mapsto |\mathrm{Li}_3(e^{it})|$ is decreasing on the interval $(0, \pi)$. Finally, as we have
\begin{align*}
    \mathrm{Re}\left(\sqrt[3]{\mathrm{Li}_3(e^{it})}\right) = \left| \mathrm{Li}_3(e^{it}) \right|^{\frac13} \cos\left( \frac13 \mathrm{Arg}(\mathrm{Li}_3(e^{it})) \right),
\end{align*}
the function on the left hand side is strictly decreasing as the product of a strictly decreasing function on the interval $(0, \pi)$.  
\end{proof}

The following corollary determines the major arcs for the proof of Theorem \ref{T:PPasymp}.

\begin{corollary} \label{C:TrilogDominant} We have 
\begin{align*}
    L(\theta) := \max_{k \geq 1} \frac{\mathrm{Re}\left( \sqrt[3]{\mathrm{Li}_3\left( e^{2\pi i k \theta}\right)}\right)}{k} = \begin{cases} \mathrm{Re}\left(\sqrt[3]{\mathrm{Li}_3(e^{2\pi i \theta})}\right), & \qquad 0 \leq \theta \leq \theta_{12}, \\ \frac{\mathrm{Re}\left(\sqrt[3]{\mathrm{Li}_3(e^{4\pi i \theta})}\right)}{2}, & \qquad \theta_{12} < \theta \leq \frac12,  \end{cases}
\end{align*}
where $\theta_{12} = 0.47585 \ldots$ is the unique solution of $\mathrm{Re}\left(\sqrt[3]{\mathrm{Li}_3(e^{2\pi i \theta})}\right) = \frac{\mathrm{Re}\left(\sqrt[3]{\mathrm{Li}_3(e^{4\pi i \theta})}\right)}{2}$ in the interval $[0,\frac12]$. Additionally, we have $L(\frac13) > L(\theta)$ for all $\frac13 < \theta \leq \frac12$.
\end{corollary}
\begin{proof} Put $f_k(\theta) := \frac{\mathrm{Re}(\sqrt[3]{\mathrm{Li}_3(e^{2\pi i k \theta})})}{k}$. By Proposition \ref{P:trilogdecrease} we see with 
\begin{align*}
    f_1\left(\frac{1}{4}\right) = 0.8391145 \ldots > f_2(0) = \frac{\sqrt[3]{\zeta(3)}}{2} = 0.531632 \ldots \geq f_k(0)
\end{align*}
that the corollary holds for $0 \leq \theta \leq \frac14$. Now we consider the case $\frac14 \leq \theta \leq \frac12$. By symmetry we argue that $f_2(\theta)$ is increasing in the interval $[\frac14, \frac12]$, so the equation $f_1(\theta) = f_2(\theta)$ only has one solution $\frac14 < \theta_{12}$ in the interval $[0,\frac12]$. It is clear that $f_2(\theta) > f_1(\theta)$ for $\theta_{12} < \theta \leq \frac12$ and $f_2(\theta) < f_1(\theta)$ for $\frac14 \leq \theta \leq \theta_{12}$. Hence, we are done after showing $f_1(\theta_{12}) > f_k(0)$ for all $k \geq 3$. This can be checked numerically, we have $f_1(\theta_{12}) = 0.5212\ldots$, but on the other hand
\begin{align*}
    f_3(0) = \frac{\sqrt[3]{\zeta(3)}}{3} = 0.3544 \ldots.  
\end{align*}
Finally, note that $L(\frac13) = 0.7304\dots$. For $\theta \in (\frac13, \theta_{12})$ the assertion $L(\frac13) > L(\theta)$ is clear by monotonicity. For $\theta \in [\theta_{12}, \frac{1}{2}]$ it follows by monotonicity and $L(\frac12) = 0.5316 \cdots < L(\frac13)$. 
\end{proof}

We now make the same considerations as above for the case of genuine regular plane partitions. The decisive function here is of type 
$$
\theta \mapsto \mathrm{Re}\left(\left(\zeta(3)-\mathrm{Li}_3\left(e^{i\theta}\right)\right)^{\frac{1}{3}}\right).
$$
\begin{lemma}\label{L:argzeta3trilog} The function 
$$
\theta \mapsto \mathrm{Arg}\left( \zeta(3) - \mathrm{Li}_3\left( e^{i\theta} \right) \right)
$$
is increasing on $(0, \pi)$ and surjects onto $(-\frac{\pi}{2},0)$.
\end{lemma}
\begin{proof} As $\mathrm{Li}_3(z)$ is a convex mapping, the curve $(0,2\pi) \to \C$, $\theta \mapsto \mathrm{Li}_3(e^{i \theta})$ bounds a convex set. Starting at $\zeta(3)$, the curve runs with monotonically increasing argument over the upper half plane to the point $-\frac34 \zeta(3)$ (halfway at $\theta=\pi$). As a result, the curve $(0,\pi) \to \C$, $\theta \mapsto \zeta(3) - \mathrm{Li}_3(e^{i \theta})$ starts at the origin\footnote{Note that the fact that the origin is contained in the closure of the image is necessary.} and runs with monotonically increasing argument over the lower half plane to the point $\frac74 \zeta(3)$. Finally, we need to show that the map surjects. It is clear that 
\begin{align*}
    \lim_{\theta \to \pi^-} \mathrm{Arg}\left(\left( \zeta(3) - \mathrm{Li}_3(e^{i\theta})\right)\right) = 0.
\end{align*}
It remains to show that
$$
\lim_{\theta \to 0^+} \mathrm{Arg}\left(\left( \zeta(3) - \mathrm{Li}_3(e^{i\theta})\right)\right) = -\frac{\pi}{2}.
$$
We compute for $\theta$ near 0,
$$
\mathrm{Arg}\left(\left( \zeta(3) - \mathrm{Li}_3(e^{i\theta})\right)\right)=\arctan\left(\frac{-\sum_{n \geq 1} \frac{\sin(n\theta)}{n^3} }{\zeta(3)-\sum_{n \geq 1} \frac{\cos(n\theta)}{n^3}} \right).
$$
As $\theta \to 0$, we have $\sum_{n \geq 1} \frac{\sin(n\theta)}{n^3} \sim \zeta(2)\theta$ and from Theorem \ref{T:EulerMac} (choosing $\vartheta = 0$), it follows that
$$
\sum_{n \geq 1} \frac{\cos(n\theta)}{n^3} \sim \zeta(3)-\frac{\theta^2 \log\left(\frac{1}{\theta}\right)}{2},
$$
whence
$$
\lim_{\theta \to 0^+} \mathrm{Arg}\left(\left( \zeta(3) - \mathrm{Li}_3(e^{i\theta})\right)\right)= \lim_{\theta \to 0^+} \arctan\left(-\frac{2\zeta(2)}{\theta \log\left(\frac{1}{\theta}\right)}\right)= -\frac{\pi}{2},
$$
proving the lemma.

\end{proof}
\begin{proposition}\label{P:zeta3trilogdecreasing}
The function
$$
\theta \mapsto \mathrm{Re}\left(\left(\zeta(3)-\mathrm{Li}_3\left(e^{i\theta}\right)\right)^{\frac{1}{3}}\right)
$$ 
is strictly increasing on $(0,\pi)$.
\end{proposition}

\begin{proof} First we show that the sequence 
\begin{align*}
    A_n := \frac{\zeta(3)}{n^3} - \sum_{k=1}^\infty \frac{1}{k^3(k+n)^3} = \sum_{k=1}^\infty \frac{1}{k^3} \left( \frac{1}{n^3} - \frac{1}{(k+n)^3}\right)
\end{align*}
is 4-fold monotone for $n \in \N$. We find
\begin{align*}
    \frac{1}{n^3} - \frac{1}{(k+n)^3} = \frac{1}{2} \int_0^1 \left( x^{n-1} - x^{k + n - 1}\right) \log(x)^2 dx.
\end{align*}
Hence, for all $m \in \N_0$
\begin{align*}
    \Delta^m \left(\frac{1}{n^3} - \frac{1}{(k+n)^3}\right) & = \frac{1}{2} \int_0^1 \sum_{j=0}^m \binom{m}{j} (-1)^j \left( x^{n+j-1} - x^{k + n + j - 1}\right) \log(x)^2 dx \\
    & = \frac{1}{2} \int_0^1 \left( x^{n-1} - x^{k+n-1}\right) \sum_{j=0}^m \binom{m}{j} (-1)^j x^j \log(x)^2 dx \\
    & = \frac{1}{2} \int_0^1 x^{n-1} \left( 1 - x^{k}\right) \left( 1 - x\right)^m \log(x)^2 dx > 0.
\end{align*}
We conclude, that the $A_n$ are 4-fold monotone. Now, we have 
\begin{align*}
    \left| \zeta(3) - \mathrm{Li}_3\left( e^{i\theta}\right)\right|^2 = \zeta(3)^2 + \zeta(6) - 2\sum_{n=1}^\infty A_n \cos(n\theta).
\end{align*}
To show that this function is increasing on $(0, \pi)$ it suffices to show that $\theta \mapsto \sum_{n=1}^\infty A_n \cos(n\theta)$ is decreasing on $(0, \pi)$, which follows from Theorem \ref{T:Fejer2} and the fact that the sequence $A_n$ is 4-fold monotone. Now we have 
\begin{align*}
    \mathrm{Re}\left( \sqrt[3]{\zeta(3) - \mathrm{Li}_3\left(e^{i\theta}\right)}\right) = \left|\zeta(3) - \mathrm{Li}_3\left(e^{i\theta}\right)\right|^{\frac{1}{3}} \cos\left( \frac{1}{3} \mathrm{Arg}\left( \zeta(3) - \mathrm{Li}_3\left(e^{i\theta}\right)\right)\right).
\end{align*}
By Lemma \ref{L:argzeta3trilog}, the cosine term is increasing, as the inner function is increasing in the interval $(-\frac{\pi}{6},0)$, so the left hand side is a product of two increasing functions. This proves the proposition. 
\end{proof}

The following proposition determines the major arc in the proof of Theorem \ref{T:planeoverpartitions}.

\begin{proposition}\label{P:zeta3trilogmax}
For any root of unity $\zeta \neq 1$ and all integers $k \geq 2$,
$$
 \mathrm{Re}\left( \frac{\left(\zeta(3)-\mathrm{Li}_3\left(\zeta^k\right)\right)^{\frac{1}{3}}}{k}\right) < \mathrm{Re}\left(\left(\zeta(3)-\mathrm{Li}_3(\zeta)\right)^{\frac{1}{3}}\right).
$$
\end{proposition}

\begin{proof}
Let $\theta_1=0.23792\dots$ be the unique solution in $(0,\pi)$ to
$$
\frac{(7\zeta(3))^{\frac{1}{3}}}{2^{\frac{5}{3}}} = \mathrm{Re}\left(\left(\zeta(3)-\mathrm{Li}_3\left(e^{i\theta}\right)\right)^{\frac{1}{3}}\right),
$$
which exists by Proposition \ref{P:zeta3trilogdecreasing}.  Then for $\theta \in (\theta_1,\pi]$, we have
$$
\mathrm{Re}\left( \frac{\left(\zeta(3)-\mathrm{Li}_3\left(e^{ik\theta}\right)\right)^{\frac{1}{3}}}{k}\right)\leq  \mathrm{Re}\left( \frac{\left(\zeta(3)-\mathrm{Li}_3\left((-1)\right)\right)^{\frac{1}{3}}}{k}\right)= \frac{(7\zeta(3))^{\frac{1}{3}}}{2^{\frac{2}{3}}k}<\mathrm{Re}\left(\left(\zeta(3)-\mathrm{Li}_3\left(e^{i\theta}\right)\right)^{\frac{1}{3}}\right).
$$

Now let $\theta \in (0,\theta_1]$ and write
$$
r_{\theta}e^{3i\psi_{\theta}}:=\zeta(3)-\mathrm{Li}_3(e^{i\theta}).
$$
We want to show
$$
\frac{r_{k\theta}^{\frac{1}{3}}\cos(\psi_{k\theta})}{k} < r_{\theta}^{\frac{1}{3}}\cos(\psi_{\theta}), 
$$
or equivalently,
\begin{equation}\label{E:cosandrtheta}
\frac{\cos(\psi_{k\theta})}{\cos(\psi_{\theta})}< k\left(\frac{r_{\theta}}{r_{k\theta}} \right)^{\frac{1}{3}}.
\end{equation}
To show the above, we bound the right-hand side from below and the left-hand side from above.  Beginning with the right-hand side, we use the mean value theorem to find $\xi_{\theta}\in (0,\theta)$ so that
$$
\zeta(3)-\mathrm{Li}_3(e^{i\theta})=\theta \frac{d}{d x} \left(\zeta(3)-\mathrm{Li}_3(e^{ix}) \right)\rvert_{x=\theta}= -i\theta \mathrm{Li}_2(e^{i \xi_{\theta}}).
$$
Using the fact that $x \mapsto |\mathrm{Li}_2(e^{ix})|$ is decreasing (see Lemma \ref{L:Lidecreasing}), we have for $k \geq 2,$
\begin{equation}\label{E:rhsofrpsi}
k\left(\frac{r_{\theta}}{r_{k\theta}} \right)^{\frac{1}{3}}= k\left|\frac{\mathrm{Li}_2\left(e^{i\xi_{\theta}}\right)}{k\mathrm{Li}_2\left(e^{i\xi_{k\theta}}\right)} \right|^{\frac{1}{3}} \geq k^{\frac{2}{3}} \left(\frac{|\mathrm{Li}_2\left(e^{i\theta_1}\right)|}{\frac{\pi^2}{6}}\right)^{\frac{1}{3}} \geq 1.5.
\end{equation}
For the left-hand side in \eqref{E:cosandrtheta}, we first note that by Lemma \ref{L:argzeta3trilog}, $\psi_{\theta}$ is increasing and negative on $(0,\theta_1].$  To find $\lim_{\theta \to 0^+} \psi_{\theta}$, we first observe
$$
\lim_{\theta\to 0^+} \frac{\zeta(3)-\mathrm{Li}_3\left(e^{i\theta}\right)}{\theta} = \lim_{\theta\to 0^+} \left(-i\mathrm{Li}_2\left(e^{i\theta}\right)\right)= -i\frac{\pi^2}{6},
$$
from which it follows that, as $\theta \to 0^+$
$$
 3\psi_{\theta} \sim \arg\left(-i\frac{\pi^2}{6}\theta\right)= -\frac{\pi}{2} \implies \lim_{\theta \to 0^+} \psi_{\theta}= -\frac{\pi}{6}.
$$
Thus, for $\theta \in (0, \theta_1],$
$$
\frac{\cos(\psi_{k\theta})}{\cos(\psi_{\theta})} \leq \frac{1}{\cos\left(-\frac{\pi}{6}\right)}= \frac{2}{\sqrt{3}}< 1.2,
$$
which with \eqref{E:rhsofrpsi} proves \eqref{E:cosandrtheta}.  The proposition follows.
\end{proof}

\section{Proof of Theorems \ref{T:PPasymp} and \ref{T:planeoverpartitions}}\label{S:proofmain}

We now apply the circle method in the form of \cite{BFG} and \cite{Parry} (see also \cite[Ch. 5]{Andrews}).  Choose $N=\lfloor \delta n^{\frac{1}{3}} \rfloor$ where $\delta$ will be taken fixed and sufficiently small in the course of the proof.  Let $\mathcal{F}_N$ be the Farey sequence of order $N$ and let $\theta_{h,k}'$ and $\theta_{h,k}''$ be the differences between $\frac{h}{k}$ and the mediants; i.e., if $\frac{h'}{k'}, \frac{h}{k}, \frac{h''}{k''}$ are three consecutive members of $\mathcal{F}_N,$ then
$$
\theta_{h,k}':=\frac{h}{k}-\frac{h'+h}{k'+k}, \qquad  \theta_{h,k}'':=\frac{h}{k}-\frac{h+h''}{k+k''}.
$$
Let
$$
t_n:=\begin{cases} \frac{2^{\frac{1}{3}}\mathrm{Li}_3\left(\zeta_b^a\right)^{\frac{1}{3}}}{n^{\frac{1}{3}}} & \text{if $\frac{a}{b}< \theta_{12}$,} \\ \frac{\mathrm{Li}_3\left(\zeta_b^{ 2a}\right)^{\frac{1}{3}}}{ 2^{\frac{2}{3}}n^{\frac{1}{3}}} & \text{if $\theta_{12}<\frac{a}{b}\leq \frac{1}{2}$,} \end{cases}
$$
and note that $t_n$ is complex.  We apply Cauchy's integral theorem and add and subtract the main term to obtain
\begin{align*}
    T_n\left(\zeta_b^a\right)&=\sum_{\substack{\gcd(h,k)=1 \\ \frac{h}{k} \in \mathcal{F}_N}} \zeta_k^{nh} \int_{-\theta_{h,k}'}^{\theta_{h,k}''} \exp\left(\log PP\left(\zeta_b^a;e^{-t_n+2\pi i \theta}\right) + nt_n-2\pi in \theta\right)d\theta \\
    &=\sum_{\substack{\gcd(h,k)=1 \\ \frac{h}{k} \in \mathcal{F}_N}} \zeta_k^{nh} \int_{-\theta_{h,k}'}^{\theta_{h,k}''} \exp\left(\frac{\mathrm{Li}_3\left(\zeta_b^{ka}\right)}{k^3(t_n-2\pi i \theta)^2} + nt_n-2\pi in \theta + E_{h,k}(\zeta_{b}^a; t_n-2\pi i \theta)\right)d\theta,
\end{align*}
where we define the error terms,
$$
E_{h,k}(z;t):=\log PP\left(z;e^{-t}\right)-\frac{\mathrm{Li}_3\left(z^k\right)}{k^3t^2}.
$$

  Corollary \ref{C:TrilogDominant} will imply that the major arc will correspond either to $\frac{h}{k}=\frac{0}{1}$ or $\frac{h}{k}=\frac{1}{2}$ depending on whether $\frac{a}{b}<\theta_{12}$ or $\theta_{12}<\frac{a}{b}<\frac{1}{2}$, respectively.  We need to estimate $E_{h,k}$ up to $o(1)$ on the major arcs and bound $E_{h,k}$ by $O(N^2)$ on the minor arcs.  The former is a straightforward application of Euler--Maclaurin summation, while the latter is one of the core difficulties we overcome in this article following the techniques of \cite{BFG}.
  
  \subsection{Bounds for the error terms}
  
\begin{lemma}\label{L:Ehk}
\begin{enumerate}
    \item If $\frac{a}{b}<\theta_{12},$ then uniformly for $\theta \in [-\theta_{0,1}',\theta_{0,1}'']$ we have
    $$
    E_{0,1}\left(\zeta_b^a;t_n-2\pi i \theta\right)=\frac{1}{12}\log\left(1-\zeta_b^a\right)+O\left(\frac{1}{n^{\frac{2}{3}}}\right).
    $$
    \item Uniformly for $\theta \in [-\theta_{0,1}',\theta_{0,1}'']$ we have
    $$
    E_{0,1}\left(1;t_n-2\pi i \theta\right)=\frac{1}{12}\log\left(t_n-2\pi i \theta\right)+\zeta'(-1)+O\left(\frac{1}{n^{\frac{2}{3}}}\right).
    $$
    \item If $\theta_{12}<\frac{a}{b}<\frac{1}{2},$ then uniformly for $\theta \in [-\theta_{1,2}',\theta_{1,2}'']$ we have
    $$
    E_{1,2}\left(\zeta_b^a;t_n-2\pi i \theta\right)=\frac{1}{6}\log\left(1-\zeta_b^a\right)-\frac{1}{12}\log\left(1+\zeta_b^a\right)+O\left(\frac{1}{n^{\frac{2}{3}}}\right).
    $$
    \item Uniformly for $\theta \in [-\theta_{1,2}',\theta_{1,2}'']$, we have
    $$
    E_{1,2}(-1;t_n-2\pi i \theta)=-\frac{1}{12}\log(t_n-2\pi i \theta) -\zeta'(-1) + +O\left(\frac{1}{n^{\frac{2}{3}}}\right).
    $$
    \item Uniformly for $k \leq N$ and $\theta \in [-\theta_{h,k}', \theta_{h,k}'']$, we have
$$
E_{h,k}\left(\zeta_{b}^a;t_n-2\pi i \theta\right) = O\left(N^2\right).
$$
\end{enumerate}
\end{lemma}

We begin by rewriting $E_{h,k}$.  
\begin{lemma}\label{L:Ehkrewrite}
For $\mathrm{Re}(t)>0$
\begin{align*}
    E_{h,k}\left(\zeta_b^a;t \right) = \sum_{\substack{1\leq j \leq k \\ 1 \leq m \leq bk}} \zeta_{b}^{ma} \zeta_{k}^{mjh} k^2 t \sum_{\ell \geq 0} g_{j,k}\left( t \left( bk^2 \ell + km \right)\right),
\end{align*}
where 
\begin{align*}
    g_{j,k}(w) := \frac{e^{-w-\frac{j}{k}w}}{w(1-e^{-w})^2} + \frac{j}{k} \frac{e^{-\frac{j}{k}w}}{w(1-e^{-w})} - \frac{1}{w^3}.
\end{align*}
\end{lemma}

\begin{proof}
Taking logarithms in \eqref{E:twovargenfn} gives
\begin{align*}
    \Log\left(PP(\zeta_b^a;\zeta_k^h e^{-t})\right) & = - \sum_{\nu \geq 1} \nu \Log\left(1-\zeta_b^a \zeta_k^{h\nu} e^{-\nu t}\right) = \sum_{\nu \geq 1} \sum_{\ell \geq 1} \nu \frac{\zeta_b^{a\ell}\zeta_k^{h\ell \nu} e^{-\nu \ell t}}{\ell}.
\end{align*}
Letting $\ell \mapsto bk\ell + m$ and $\nu \mapsto \nu k +j$ we obtain
\begin{align*}
   \Log\left(PP(\zeta_b^a;\zeta_k^h e^{-t})\right)
      = & \sum_{\substack{1\leq j \leq k \\ 1 \leq m \leq bk}} \zeta_{b}^{ma} \zeta_{k}^{mjh} \sum_{\ell \geq 0} \frac{e^{-j(bk\ell + m)t}}{bk\ell+m} \sum_{\nu \geq 0} (\nu k +j) e^{-\nu k (bk\ell +m) t} \\
      = & \sum_{\substack{1\leq j \leq k \\ 1 \leq m \leq bk}} \zeta_{b}^{ma} \zeta_{k}^{mjh} k^2 t \sum_{\ell \geq 0} \frac{e^{-j(bk\ell + m)t}}{k(bk\ell+m)t} \left( \frac{e^{-k(bk\ell + m)t}}{\left(1 - e^{-k(bk\ell + m)t}\right)^2} + \frac{j}{k} \frac{1}{1 - e^{-k(bk\ell + m)t}}\right).
\end{align*}

  For the right term in $g_{j,k}$, we compute (using $\gcd(h,k)=1$)
  \begin{align*}
      - \sum_{\substack{1 \leq j \leq k \\ 1 \leq m \leq bk}} \zeta_{b}^{ma} \zeta_{k}^{mjh} k^2 t \sum_{\ell \geq 0} \frac{1}{t^3(bk^2\ell +mk)^3}  &=-\frac{1}{t^2}\sum_{\substack{1 \leq m \leq bk \\ m \equiv 0 \pmod{k}}} \sum_{\ell\geq 0}\frac{\zeta_{b}^{ma}}{(bk\ell+m)^3} \\
      &=-\frac{1}{t^2} \sum_{\ell \geq 0} \sum_{\nu=1}^b \frac{\zeta_b^{\nu k a}}{(bk\ell+ k\nu)^3} \\
      &=-\frac{1}{k^3t^2} \mathrm{Li}_3\left(\zeta_b^{ak} \right). \qedhere
  \end{align*}
\end{proof}

\begin{proof}[Proof of Lemma \ref{L:Ehk}]
 Throughout this section we write $t_{\theta}:=t_n-2 \pi i \theta.$  Recall (see \cite[Ch. 5]{Andrews}) that \newline $\theta_{h,k}',\theta_{h,k}'' \asymp \frac{1}{kN} \ll \frac{1}{n^{\frac{1}{3}}}$.  It follows that for some cone $D_{\psi}$ described in Theorem \ref{T:EulerMac}, we have $t_{\theta}\in D_{\psi}$ for all $\theta \in [\theta_{h,k}',\theta_{h,k}'']$.  Furthermore, as $w \to 0$, we have
$$ g_{j,k}(w) = \left(-\frac{j^2}{2k^2}-\frac{1}{12}+\frac{j}{2k}\right)w^{-1}+O(1),$$ and thus Theorem \ref{T:EulerMac} applies to those $\theta$ in Lemma \ref{L:Ehk}.  Hence with $$ b_{-1}=b_{-1}\left(\frac{j}{k}\right):= -\frac{j^2}{2k^2}-\frac{1}{12}+\frac{j}{2k}, $$
we have
\begin{align}\label{E:gEulerMac}
\sum_{\ell\geq 0}g_{j,k}(t_{\theta}(bk^2\ell+km)) = \frac{b_{-1}\Log\left(\frac{1}{bk^2t_{\theta}}\right)}{bk^2t_{\theta}}+\frac{b_{-1}C_{\frac{m}{bk}}}{bk^2t_{\theta}} + \frac{1}{bk^2t_{\theta}}\int^\infty_{0}\left(g_{j,k}(x)-\frac{b_{-1}e^{-x}}{x}\right)dx+O(t_{\theta}k^2),
\end{align}
with $$C_{\frac{m}{bk}}= \left(1-\frac{m}{bk}\right)\sum_{r\geq 0}\frac{1}{(r+\frac{m}{bk})(r+1)}= O(1). $$

Here, $b_{-1}(1)=-\frac{1}{12}$, so for part (1) we have
\begin{align*}
E_{0,1}\left(\zeta_b^a;t_{\theta}\right)&= \sum_{1 \leq m \leq b} \zeta_b^{ma} t_{\theta} \left(-\frac{\Log\left(\frac{1}{bt_{\theta}}\right)}{12bt_{\theta}} - \frac{C_{\frac{m}{b}}}{12bt_{\theta}}+\frac{1}{bt_{\theta}}\int_0^{\infty}\left(g_{1,1}(x)+\frac{e^{-x}}{12x}\right)dx + O\left(t_{\theta} \right) \right) \\
&= -\frac{1}{12b}\sum_{1 \leq m \leq b} \zeta_b^{ma}C_{\frac{m}{b}} +O\left(t_{\theta}^2\right) \\
&=\frac{1}{12}\log\left(1-\zeta_b^a\right)+O\left(\frac{1}{n^{\frac{2}{3}}}\right),
\end{align*}
as claimed.  The proof of parts (2)-(4) are similar.
\\
\

For part (5), we follow \cite{BFG} and separate the cases for small and large $k$.  First let $k \leq n^{\frac{1}{6}}$ and define
$$
S_{a,b}:=\{(h,k)\in \mathbb{N}^2 : ak+bjh \equiv 0 \pmod{bk}, \ \text{for some $1\leq j \leq k$}\}.
$$
It is easy to see that if $(h,k) \in S_{a,b}$ then there is only one corresponding $j$; let this be $j_0=j_0(h,k).$  Thus, plugging \eqref{E:gEulerMac} into Lemma \ref{L:Ehkrewrite}, we obtain
\begin{align*}
    E_{h,k}(\zeta_b^a;t_{\theta})&=\sum_{\substack{1\leq j \leq k \\ j \neq j_0 \\ 1 \leq m \leq bk}} \frac{\zeta_b^{ma}\zeta_k^{mjk}}{b}\left(b_{-1}\left(\frac{j}{k}\right)C_{\frac{m}{bk}} +O(t_{\theta}k^2)\right)  + 1_{(h,k) \in S_{a,b}}\left( kb_{-1}\left(\frac{j_0}{k}\right)\Log\left(\frac{1}{bk^2t_{\theta}}\right) \right. \\ & \qquad  \left. +kb_{-1}\left(\frac{j_0}{k}\right)  +k\int^\infty_{0}\left(g_{j_0,k}(x)-\frac{b_{-1}\left(\frac{j_0}{k}\right)e^{-x}}{x}\right)dx+O(t_{\theta}k^3)\right) \\ 
    &= O\left(k^2+t_{\theta}k^4\right),
\end{align*}
and this is certainly $O(N^2)$.

Now suppose that $n^{\frac{1}{6}}\leq k \leq N$.  We consider two separate cases for the sums in Lemma 4.2, $\ell = 0$ and $\ell \geq 1$. First let $\ell=0$. Note that we have $\phi_{\frac{j}{k}}(z) = zg_{j,k}(z)$, where $\phi_a(z)$ was defined in \eqref{D:phia}. Hence, we find
\begin{align*}
\sum_{\substack{1\leq j \leq k \\ 1 \leq m \leq bk}} \zeta_{b}^{ma} \zeta_{k}^{mjh} k^2 t g_{j,k}\left( ktm\right) & = k \sum_{\substack{1 \leq j \leq k \\ 1 \leq m \leq bk}} \zeta_{b}^{ma} \zeta_{k}^{mjh} \frac{ \phi_{\frac{j}{k}}(ktm)}{m}.
\end{align*}
We now adapt the methodology of \cite{BFG} and prove the following. 

\begin{align}\label{E:psibound}
k \sum_{\substack{1 \leq j \leq k \\ 1 \leq m \leq bk}} \frac{\zeta_{b}^{ma} \zeta_{k}^{mjh}}{m} \phi_{\frac{j}{k}}(k t_{\theta } m) = O\left(k^2\right).
\end{align}

Put $x := kt_{\theta}$ and $a := \frac{j}{k}$. Note that we have $0 < a \leq 1$. We have $\mathrm{Re}(x) > 0$ and $\frac{|x|}{\mathrm{Re}(x)} \ll 1$ uniformly in $k$. We use partial summation and split the sum into two parts. 
    \begin{align*}
        \sum_{\substack{1 \leq m \leq bk \\ 1 \leq j \leq k}} = \sum_{1 \leq j \leq k} \left( \sum_{0 < m \leq \min\left\{bk, \frac{1}{|x|}\right\}} + \sum_{\min\left\{bk, \frac{1}{|x|}\right\} < m \leq bk}\right).
    \end{align*}
    In the case $|x| > 1$, the first sum is empty, so we can assume $|x| \leq 1$. We first find with Proposition \ref{P:Abelpartialsummation} 
    \begin{align*}
        & \sum_{0 < m \leq \min\left\{bk, \frac{1}{|x|}\right\}}  \zeta_{bk}^{m(ak+hjb)}\frac{1}{m}\phi_a(xm) = G_{\min\left\{bk, \frac{1}{|x|}\right\}}\left( \frac{a}{b} + \frac{hj}{k} \right) \phi_a\left( x \min\left\{bk, \Big\lfloor \frac{1}{|x|} \Big\rfloor \right\} \right) \\
        & + \sum_{m \leq \min\left\{bk, \frac{1}{|x|}\right\} - 1} G_{m}\left( \frac{a}{b} + \frac{hj}{k} \right) \left( \phi_a(mx) - \phi_a((m+1)x)\right).
    \end{align*}
    It follows with Lemma \ref{SmallRadiusEstimate}
    \begin{align*}
    & k\left| \sum_{1 \leq j \leq k} \sum_{m \leq \min\left\{bk, \frac{1}{|x|}\right\}} \zeta_b^{ma}\zeta_k^{mhj}\cdot \frac{1}{m}\phi_a(mx) \right| \\
    & \leq k\sum_{1 \leq j \leq k} \left|G_{\min\left\{bk, \frac{1}{|x|}\right\}}\left( \frac{a}{b} + \frac{hj}{k} \right)\right| \left| \phi_a\left( x \min\left\{bk, \Big\lfloor\frac{1}{|x|}\Big\rfloor \right\}  \right) \right| \\
    & \hspace{3cm} + k\left| \sum_{1 \leq j \leq k} \sum_{0 < m \leq \min\left\{bk, \frac{1}{|x|}\right\}} G_{m}\left( \frac{a}{b} + \frac{hj}{k} \right) \left( \phi_a(mx) - \phi_a((m+1)x)\right) \right| \\
    & \ll k\sum_{1 \leq j \leq k} \left|G_{\min\left\{bk, \frac{1}{|x|}\right\}}\left( \frac{a}{b} + \frac{hj}{k} \right)\right| + k\sum_{1 \leq j \leq k} \max_{1\leq m \leq  \min\left\{bk, \frac{1}{|x|}\right\}} \left|G_{m}\left( \frac{a}{b} + \frac{hj}{k} \right)\right| \sum_{0 < m \leq \frac{1}{|x|}} |x| = O\left(k^2\right),
    \end{align*}
    where we have used Lemma \ref{L:Gmaxbound} and $G_{bk}(1) = H_{bk} = O(\log(k))$ in the last step. Similarly, we find with Lemma \ref{LargeRadiusEstimate} (without loss of generality we assume $\frac{1}{|x|} < bk$)
    \begin{align*}
        & \left| \sum_{1 \leq j \leq k} \sum_{\frac{1}{|x|} < m \leq bk} \zeta_b^{ma}\zeta_{k}^{mhj}\cdot \frac{1}{m}\phi_a(mx) \right| \\
        & \ll \sum_{1 \leq j \leq k} \left| G_{bk}\left( \frac{a}{b} + \frac{hj}{k}\right) \right| \left| \phi_a\left( x bk \right) \right| + \left| \sum_{1 \leq j \leq k} \sum_{\frac{1}{|x|} < m \leq bk} G_{m}\left( \frac{a}{b} + \frac{hj}{k} \right) \left( \phi_a(mx) - \phi_a((m+1)x)\right) \right| \\
        & \ll O(k) + \sum_{1 \leq j \leq k} \max_{\frac{1}{|x|}\leq m \leq  bk} \left|G_{m}\left( \frac{a}{b} + \frac{hj}{k} \right)\right| \\
        & \hspace{1cm} \times \sum_{\frac{1}{|x|} < m \leq bk} \left( \frac{1}{|x|^2} \left( \frac{1}{m^2} - \frac{1}{(m+1)^2}\right) + |x|e^{-m\mathrm{Re}(x)} + |x| e^{-m\mathrm{Re}(x)(1+a)} + a^2|x| e^{-am\mathrm{Re}(x)}\right). 
    \end{align*}
    Note that we uniformly have $\phi_a(xbk) \ll 1$ (as $1 \ll |xbk|$ and $x$ is part of a fixed cone $|\Arg(x)|\leq \frac{\pi}{2}-\eta$) as well as 
    \begin{align*}
    \sum_{\frac{1}{|x|} < m \leq bk} \frac{1}{|x|^2}\left(\frac{1}{m^2} - \frac{1}{(m+1)^2}\right) & \leq \sum_{\frac{1}{|x|} < m < \infty} \frac{1}{|x|^2}\left(\frac{1}{m^2} - \frac{1}{(m+1)^2}\right) \leq \frac{1}{|x|^2 \frac{1}{|x|^2}} = 1, \\
    \sum_{\frac{1}{|x|} < m \leq bk} |x| e^{-m \mathrm{Re}(x)} & \leq \frac{|x|}{1-e^{-\mathrm{Re}(x)}} \ll 1, \qquad \text{as} \quad \frac{|x|}{\mathrm{Re}(x)} \ll 1, |x| \ll 1, \\
    \intertext{and similarly}
    \sum_{\frac{1}{|x|} < m \leq bk} |x| e^{-m \mathrm{Re}(x)(a+1)} & \leq \frac{|x|}{1-e^{-\mathrm{Re}(x)(a+1)}} \ll 1.
\end{align*}
    As a result, using Lemma \ref{L:Gmaxbound} (again up to at most one summand in $O(\log(k))$),
        \begin{align*}
        \left| k\sum_{1 \leq j \leq k} \sum_{\frac{1}{|x|} < m \leq bk} \zeta_b^{ma}\zeta_{k}^{mhj}\cdot \frac{1}{m}\phi_a(mx) \right| & \ll O\left(k^2\right) + k\sum_{1 \leq j \leq k} \max_{\frac{1}{|x|}\leq m \leq  bk} \left|G_{m}\left( \frac{a}{b} + \frac{hj}{k} \right)\right| = O\left(k^2\right),
        \end{align*}
        which yields \eqref{E:psibound}.
        
        For the summands $\ell \geq 1$ in Lemma \ref{L:Ehkrewrite}, we 
       first focus on the $j$ sum,

\begin{align*}
    \sum_{1\leq j\leq k}\zeta_{k}^{mjh}g_{j,k}\left( w\right) &=  \sum_{1\leq j\leq k}\zeta_{k}^{mjh}\left(\frac{e^{-w-\frac{jw}{k}}}{w(1-e^{-w})^2}+\frac{j}{k}\frac{e^{-\frac{jw}{k}}}{w(1-e^{-w})}-\frac{1}{w^3} \right)\\
    &= -\frac{\mathrm{1}_{k|mh}}{w^3}+\frac{\zeta_k^{mh}e^{-w-\frac{w}{k}}}{w(1-e^{-w})^2}\left(\frac{1-e^{-w}}{1-\zeta_k^{mh}e^{-\frac{w}{k}}}\right)\\
    &\hspace{10mm}+\frac{\zeta_k^{mh}e^{-\frac{w}{k}}}{kw(1-e^{-w})}\left(\frac{1+k\zeta_k^{mh}e^{-\frac{w}{k}-w}-(k+1)e^{-w}}{(1-\zeta_k^{mh}e^{-\frac{w}{k}})^2}\right),
\end{align*}
where we used the formulas $$\sum^k_{j=1}a^j = \frac{1-a^{k+1}}{1-a}-1, \qquad \sum^k_{j=1}ja^j = a\frac{\left(ka^{k+1}+1-(k+1)a^k\right)}{1-a^2}.$$ 
When $k$ divides $m$, we have 
\begin{align*}
    \sum_{\substack{1\leq j\leq k}}\zeta_{k}^{mjh}g_{j,k}\left( w\right) &= -\frac{1}{w^3}+\frac{e^{-w-\frac{w}{k}}}{w(1-e^{-w})(1-e^{-\frac{w}{k}})}+\frac{e^{-\frac{w}{k}}\left(1-e^{-w}+ke^{-w}(-1+e^{-\frac{w}{k}}) \right)}{kw(1-e^{-w})(1-e^{-\frac{w}{k}})^2}\\
    &= -\frac{1}{w^3}+\frac{e^{-\frac{w}{k}}}{kw(1-e^{-\frac{w}{k}})^2}=: f_1(w),
\end{align*}
and when $k$ does not divide $m$ we have 
\begin{align*}
     \sum_{\substack{1\leq j\leq k}}\zeta_{k}^{mjh}g_{j,k}\left( w\right) &= \frac{\zeta_k^{mh}e^{-w-\frac{w}{k}}}{w(1-e^{-w})(1-\zeta^{mh}_ke^{-\frac{w}{k}})} + \frac{\zeta_k^{mh}e^{-\frac{w}{k}}}{kw(1-e^{-w})}\left(\frac{1-e^{-w}+k\zeta_k^{mh}e^{-w-\frac{w}{k}}-ke^{-w}}{(1-\zeta_k^{mh}e^{-\frac{w}{k}})^2}\right)\\
     &= \frac{\zeta^{mh}_ke^{-\frac{w}{k}}}{kw(1-\zeta_k^{mh}e^{-\frac{w}{k}})^2}=: f_2(w).
\end{align*}

We will finish the proof of Lemma \ref{L:Ehk} by showing
\begin{equation}\label{E:f1bound}
k^2t\sum_{\substack{m\leq bk \\ k \mid m}}\sum_{\ell\geq 1}\zeta_{b}^{ma}f_{1}\left( t \left( bk \ell + m \right)\right) =O\left(k^2\right),
\end{equation}
and 
\begin{equation}\label{E:f2bound}
k^2t\sum_{\substack{ m\leq bk \\ k \nmid m}}\sum_{\ell\geq 1}\zeta_{b}^{ma}f_{2}\left( t k\left( bk \ell + m \right)\right) =O\left(k^2\right).
\end{equation}

First, we apply Theorem \ref{T:EulerMacclassic} to the sum on $f_1$.  Computing the derivative of $f_1(w)$, we find
\begin{align*}
    f^{'}_{1}(w) = 3\,{w}^{-4}-{\frac {1}{{k}^{2}w}{{\rm e}^{-{\frac {w}{k}}}} \left( 1-{
e}^{-{\frac {w}{k}}} \right) ^{-2}}-{\frac {1}{k{w}^{2}}{{\rm e}^{-{
\frac {w}{k}}}} \left( 1-{e}^{-{\frac {w}{k}}} \right) ^{-2}}-2\,{
\frac {1 }{{k}^{2}w}{{\rm e}^{-{\frac {w}{k}}}}{e}
^{-{\frac {w}{k}}} \left( 1-{e}^{-{\frac {w}{k}}} \right) ^{-3}}.
\end{align*}
Thus,

\begin{align*}
\sum_{\ell\geq 1}f_{1}\left( t \left( bk \ell + m \right)\right) &= \int^\infty_{1}f_1(t_\theta(bkw+m))dw +\frac{\displaystyle{\lim_{\ell\to \infty}f_1(t_\theta(bk\ell+m))}+f_1(t_\theta(bk+m))}{2}\\
&\hspace{10mm}+ t_{\theta}bk\int^\infty_{1}f_1^{'}(t_\theta(bkw+m))\left(\{w\}-\frac{1}{2}\right)dw\\
&= \frac{1}{t_{\theta}bk}\int^\infty_{t_{\theta}(bk+m)}f_1(x))dx + \frac{f_1(t_\theta(bk+m))}{2}\\
&\hspace{10mm}+ \int^\infty_{t_{\theta}(bk+m)}f_1^{'}(x)\left(\left\{ \frac{x-t_\theta m}{t_\theta b k}\right\}-\frac{1}{2}\right)dw.
\end{align*}
We note that $f_1(w)\ll \frac{k}{w^3}$ as $w\to 0$ in a cone, which comes from the Laurent expansion
$$  f_1(w) = \left( -1+k
 \right) {w}^{-3}+ 
  -\frac{7}{12}w^{-1}+O(1).
$$ Since $|t_{\theta}|  \asymp \frac{1}{n^{\frac{1}{3}}}$ and $k \ll n^{\frac{1}{3}}$, there exists a fixed complex number $c$ on the line of integration such that $|t_{\theta}(bk+m)|< |c|.$ Then $\int^\infty_c f_1(w) dw =O_c(1)$. Near $0$,  we have
$$ \int^c_{t_{\theta}(bk+m)}f_1(w)dw\ll \int^{c}_{t_{\theta}(bk+m)}\frac{1}{w^3}dw\ll O\left(\frac{1}{t_{\theta}^2(bk+m)^2}\right) = O\left(\frac{1}{t_\theta^2 k^2}\right),$$
where the last step follows from the fact that $k | m$. Therefore,
$$k^2t_{\theta}\sum_{k|m\leq bk}\zeta_{b}^{ma}\frac{1}{t_\theta k b}\int^\infty_{1}f_1(t_\theta(bkw+m))dw= O\left(\frac{1}{t_\theta^2 k}\right)\ll N^{\frac{3}{2}}.$$

We have $f_1(t_\theta(bk+m))\ll \frac{1}{t_\theta^3 k^3}$ which implies that 
$$ k^2t_{\theta}\sum_{k|m\leq bk}\zeta_{b}^{ma}f_1(t_\theta(bk+m))\ll n^{1/2} \ll N^{\frac{3}{2}}.$$ 
It remains to contend with the integral involving the derivative. The derivative has the Laurent expansion
\begin{align*}
    f^{'}_1(w) = 3(1-k)w^{-4}+O(w^{-2})\ll \frac{k}{w^4}.
\end{align*}
Thus, breaking the integral as before into two parts (with $c$ fixed) we find
$$k^2t_{\theta}\sum_{k|m\leq bk}\zeta_{b}^{ma}\int^\infty_{1}f^{'}_1(t_\theta(bkw+m))dw = k^3t_{\theta}\frac{1}{t^2_\theta k^2}\ll N^{\frac{3}{2}}.$$
This completes the proof of \eqref{E:f1bound}.

Now turning to $\eqref{E:f2bound},$ it is clear that in the right half plane
$$ \left| f_2(w)\right|\leq \frac{e^{-\frac{\mathrm{Re}(w)}{k}}}{k\mathrm{Re}(w)(1-e^{-\frac{\mathrm{Re}(w)}{k}})^2}:= \Tilde{f}_2(\mathrm{Re}(w)),$$

where $\Tilde{f}_2(\mathrm{Re}(w))$ is a decreasing function. Approximating by a Riemann sum, we have
\begin{align*}
    &k^2t_{\theta}\sum_{k\not| m\leq bk}\sum_{\ell\geq 1}\zeta_{b}^{ma}f_{2}\left( t_{\theta} k\left( bk \ell + m \right)\right) \leq k^3t_{\theta}\sum_{\ell\geq 1}\Tilde{f}_{2}\left( \mathrm{Re}(t_{\theta} k^2 b\ell\right)\\
    &\leq \frac{k|t_\theta|}{b\mathrm{Re}(t_\theta)}\left(\frac{1}{\mathrm{Re}(t_\theta)k^2b}\int^\infty_{\mathrm{Re}(t_\theta)k^2b}\Tilde{f}_2(x)dx+ bk^2\mathrm{Re}(t_\theta)\Tilde{f}_{2}\left(\mathrm{Re}(t_{\theta} k^2 b\right)\right).
\end{align*}
Since $k\gg n^{1/6}$, the lower bound of the integral is bounded away from $0$. Therefore,
\begin{align*}
    k^2t_{\theta}\sum_{k\not| m\leq bk}\sum_{\ell\geq 1}\zeta_{b}^{ma}f_{2}\left( t_{\theta} k\left( bk \ell + m \right)\right) &\ll O(n^{1/6})+ k^3|t_\theta|\Tilde{f}_{2}\left(\mathrm{t_{\theta}} k^2 b\right).
\end{align*}
Furthermore, 
$$\Tilde{f}_{2}\left(\mathrm{t_{\theta}} k^2 b\right)\ll \frac{1}{k^4|t_\theta|^2}\ll 1,$$
and \eqref{E:f2bound} follows, completing the proof of Lemma \ref{L:Ehk}

\end{proof}

\subsection{Proof of Theorem \ref{T:PPasymp}}
In this section we follow Parry \cite{Parry}.  We begin with the case $\frac{a}{b}<\theta_{12}.$  Write
$$
\lambda_1+i\lambda_2:=\mathrm{Li}_3\left(\zeta_b^a\right)^{\frac{1}{3}}.
$$

Then by Lemma \ref{L:Ehk},
\begin{align}
    e^{-\frac{3}{2^{\frac{2}{3}}}\lambda_1n^{\frac{2}{3}}}T_n\left(\zeta_b^a\right)&=\left(1-\zeta_b^a\right)^{\frac{1}{12}}\int_{-\theta_{0,1}'}^{\theta_{0,1}''} \exp\left(-\frac{3}{2^{\frac{2}{3}}}\lambda_1n^{\frac{2}{3}}+\frac{(\lambda_1+i\lambda_2)^3}{(t_n-2\pi i \theta)^2} + nt_n-2\pi in \theta + o(1)\right)d\theta \nonumber \\
    &+\sum_{\substack{k \geq 2 \\ \frac{h}{k} \in \mathcal{F}_N}} \zeta_k^{nh} \int_{-\theta_{h,k}'}^{\theta_{h,k}''} \exp\left(-\frac{3}{2^{\frac{2}{3}}}\lambda_1n^{\frac{2}{3}}+\frac{\mathrm{Li}_3\left(\zeta_b^{ka}\right)}{k^3(t_n-2\pi i \theta)^2} + nt_n-2\pi in \theta + \delta O(n^{\frac{2}{3}})\right)d\theta, \label{E:majorminorarcsplit}
\end{align}
where the constants are independent of $k$.  For the major arc integral above, we rewrite the exponent as
\begin{align*}
    \frac{\lambda_1+i\lambda_2}{2^{\frac{2}{3}}\left(1-\frac{2\pi i n^{\frac{1}{3}}\theta}{2^{\frac{1}{3}}(\lambda_1+i\lambda_2)}\right)^2}n^{\frac{2}{3}}-\frac{1}{2^{\frac{2}{3}}}\lambda_1n^{\frac{2}{3}}+i2^{\frac{1}{3}}\lambda_2n^{\frac{2}{3}}-2\pi i n \theta + o(1).
\end{align*}
Setting $\theta \mapsto n^{-\frac{1}{3}}\theta$ and using the fact that $\theta_{0,1}'=\theta_{0,1}''=\frac{1}{N+1} \asymp \frac{1}{n^{\frac{1}{3}}},$ the major arc integral is asymptotic to
$$
\frac{1}{n^{\frac{1}{3}}}\int_{-c}^c e^{n^{\frac{2}{3}}B(\theta)}d\theta,
$$
for some $c>0$, where
$$
B(\theta):= \frac{\lambda_1+i\lambda_2}{2^{\frac{2}{3}}\left(1-\frac{2\pi i \theta}{2^{\frac{1}{3}}(\lambda_1+i\lambda_2)}\right)^2}-\frac{1}{2^{\frac{2}{3}}}\lambda_1+i2^{\frac{1}{3}}\lambda_2-2\pi i  \theta.
$$
Here, $B(0)=i\frac{3}{2^{\frac{2}{3}}}\lambda_2$ and using $\frac{1}{(1-x)^2}=1+2x+3x^2+O(x^3),$ we have
$$
\lim_{\theta \to 0} \frac{B(\theta)-B(0)}{\theta^2}=3\frac{(2\pi i )^2}{2^{4/3}(\lambda_1+i\lambda_2)}=-\frac{2^{\frac{2}{3}} 3\pi^2}{\lambda_1+i\lambda_2}.
$$
Now we show that $\mathrm{Re}(B(\theta)) \leq \mathrm{Re}(B(0))=0$ with equality if and only if $\theta=0.$  Following \cite[Lemma 5.2]{Parry}, we write
\begin{align*}
\mathrm{Re}(B(\theta))&=\mathrm{Re}\left(\frac{\lambda_1+i\lambda_2}{2^{\frac{2}{3}}\left(1-\frac{2\pi i \theta}{2^{\frac{1}{3}}(\lambda_1+i\lambda_2)}\right)^2}\right)-\frac{\lambda_1}{2^{\frac{2}{3}}} \\
&=\frac{|\lambda_1+i\lambda_2|^3}{2^{\frac{2}{3}}\lambda_1^2}\mathrm{Re}\left(\frac{e^{i3r}}{\left(1+i\left(\frac{\lambda_2}{\lambda_1}-\frac{2^{\frac{2}{3}}\pi\theta}{\lambda_1}\right)\right)^2}\right)-\frac{\lambda_1}{2^{\frac{2}{3}}} & \text{($r:=\arg(\lambda_1+i\lambda_2)$)}.
\end{align*}
Next, we write
$$
\left|1+i\left(\frac{\lambda_2}{\lambda_1}-\frac{2^{\frac{2}{3}}\pi\theta}{\lambda_1}\right)\right|e^{i\psi_{\theta}}:=1+i\left(\frac{\lambda_2}{\lambda_1}-\frac{2^{\frac{2}{3}}\pi\theta}{\lambda_1}\right).
$$
Note that $\theta \mapsto \psi_{\theta}$ is a bijection from $\left(-\frac{\pi}{2},\frac{\pi}{2}\right)$ to $\mathbb{R}$, and further by matching real and imaginary parts of both sides, we see that
$$
\frac{1}{\left|1+i\left(\frac{\lambda_2}{\lambda_1}-\frac{2^{\frac{2}{3}}\pi \theta}{\lambda_1}\right)\right|}=\cos\left(\psi_{\theta}\right).
$$
Thus,
$$\mathrm{Re}(B(\theta))=\frac{|\lambda_1+i\lambda_2|^3}{2^{\frac{2}{3}}\lambda_1^2}\cos(3r-2\psi_{\theta})\cos^2(\psi_{\theta})-\frac{\lambda_1}{2^{\frac{2}{3}}}.
$$
By Lemma \ref{L:cosinemax}, the maximum is achieved when $3r-2\psi_{\theta}=\psi_{\theta}$, i.e. $\psi_{\theta}=r$, which gives
$$\mathrm{Re}(B(\theta))\leq \frac{|\lambda_1+i\lambda_2|^3}{2^{\frac{2}{3}}\lambda_1^2}\cos(r)^3-\frac{\lambda_1}{2^{\frac{2}{3}}}=\frac{\lambda_1^3}{2^{\frac{2}{3}}\lambda_1^2}-\frac{\lambda_1}{2^{\frac{2}{3}}}=0,
$$
with equality if and only if $\theta=0$.  Thus by Theorem \ref{T:laplacemethod}, we have
$$
\frac{1}{n^{\frac{1}{3}}}\int_{-c}^{c}e^{n^{\frac{2}{3}}B(\theta)}d\theta \sim \frac{1}{n^{\frac{1}{3}}}\sqrt{\frac{\pi}{n^{\frac{2}{3}}\frac{2^{\frac{2}{3}}3\pi^2}{\lambda_1+i\lambda_2}}}=\frac{\sqrt{\lambda_1+i\lambda_2}}{2^{\frac{1}{3}}\sqrt{3\pi}n^{\frac{2}{3}}}e^{i\frac{3}{2^{\frac{2}{3}}}\lambda_2n^{\frac{2}{3}}},
$$
and this is the main term in case (1) of Theorem \ref{T:PPasymp} after multiplying both sides again by $e^{\frac{3}{2^{\frac{2}{3}}}\lambda_1n^{\frac{2}{3}}}$ and the constant $\left(1-\zeta_b^a\right)^{1/12}$ in \eqref{E:majorminorarcsplit}.

We now show that the minor arcs in \eqref{E:majorminorarcsplit} tend to 0; taking the real part of the exponent, we get
\begin{align*}
    \frac{n^{\frac{2}{3}}\lambda_1}{2^{\frac{2}{3}}}\left(-1 + \frac{\left|\mathrm{Li}_3\left(\zeta_b^{ka}\right)\right|}{k^3\lambda_1^3}\mathrm{Re}\left(\frac{e^{i3r_k}}{\left(1+i\left(\frac{\lambda_2}{\lambda_1}-\frac{2^{\frac{2}{3}}\pi\theta n^{\frac{1}{3}}}{\lambda_1}\right) \right)^2}\right) +\delta O(1)\right)  \qquad \text{$\left(r_k:=\arg \mathrm{Li}_3\left(\zeta_b^{ak}\right)^{\frac{1}{3}}\right)$}.
\end{align*}
As before, the maximum of the expression Re$(\cdot)$ is $\cos(r_k)^3$, so we get that the above is at most
$$
\frac{n^{\frac{2}{3}}\lambda_1}{2^{\frac{2}{3}}}\left(-1 + \left(\frac{\mathrm{Re}\left(\mathrm{Li}_3\left(\zeta_b^{ka}\right)\right)}{k\lambda_1}\right)^3 +\delta O(1)\right) \leq \frac{n^{\frac{2}{3}}\lambda_1}{2^{\frac{2}{3}}}\left(-\Delta+\delta O(1)\right),
$$
for some $\Delta>0$ uniformly in $k$ by Corollary \ref{C:TrilogDominant}.  Choosing $\delta$ small enough so that the constant above is still negative, we get that the minor arcs decay exponentially, finishing the proof in case (1).

In the case that $\theta_{12}<\frac{a}{b}<\frac{1}{2}$, we set $\lambda_1+i\lambda_2:=\frac{\mathrm{Li}_3\left(\zeta_b^{2a}\right)^{\frac{1}{3}}}{2},$ and again by Lemma \ref{L:Ehk} we have
\begin{align}
    e^{-\frac{3}{2^{\frac{2}{3}}}\lambda_1n^{\frac{2}{3}}}T_n\left(\zeta_b^a\right)&=(-1)^n\frac{\left(1-\zeta_b^a\right)^{\frac{1}{6}}}{\left(1+\zeta_b^a\right)^{\frac{1}{12}}}\int_{-\theta_{1,2}'}^{\theta_{1,2}''} \exp\left(-\frac{3}{2^{\frac{2}{3}}}\lambda_1n^{\frac{2}{3}}+\frac{(\lambda_1+i\lambda_2)^3}{(t_n-2\pi i \theta)^2} + nt_n-2\pi in \theta + o(1)\right)d\theta \nonumber \\
    &+\sum_{\substack{k \neq 2 \\ \frac{h}{k} \in \mathcal{F}_N}} \zeta_k^{nh} \int_{-\theta_{h,k}'}^{\theta_{h,k}''} \exp\left(-\frac{3}{2^{\frac{2}{3}}}\lambda_1n^{\frac{2}{3}}+\frac{\mathrm{Li}_3\left(\zeta_b^{ka}\right)}{k^3(t_n-2\pi i \theta)^2} + nt_n-2\pi in \theta + \delta O(n^{\frac{2}{3}})\right)d\theta. \label{E:majorminorarcsplitcase2}
\end{align} 
Whereas for $\frac{a}{b}=\frac{1}{2}$, we set $\lambda_1:=\frac{\zeta(3)^{\frac{1}{3}}}{2},$ and Lemma \ref{L:Ehk} implies
\begin{align}
    &e^{-\frac{3}{2^{\frac{2}{3}}}\lambda_1n^{\frac{2}{3}}}T_n\left(-1\right)\nonumber \\
    & \hspace{1cm}=(-1)^ne^{-\zeta(-1)}\int_{-\theta_{1,2}'}^{\theta_{1,2}''}(t_n-2\pi i \theta)^{-\frac{1}{12}} \exp\left(-\frac{3}{2^{\frac{2}{3}}}\lambda_1n^{\frac{2}{3}}+\frac{\lambda_1^3}{(t_n-2\pi i \theta)^2} + nt_n-2\pi in \theta + o(1)\right)d\theta \nonumber \\
    &\hspace{2cm} +\sum_{\substack{k \neq 2 \\ \frac{h}{k} \in \mathcal{F}_N}} \zeta_k^{nh} \int_{-\theta_{h,k}'}^{\theta_{h,k}''} \exp\left(-\frac{3}{2^{\frac{2}{3}}}\lambda_1n^{\frac{2}{3}}+\frac{\mathrm{Li}_3\left(\zeta_b^{ka}\right)}{k^3(t_n-2\pi i \theta)^2} + nt_n-2\pi in \theta + \delta O(n^{\frac{2}{3}})\right)d\theta. \label{E:majorminorarcsplitcase3}
\end{align}
The analysis of \eqref{E:majorminorarcsplitcase2} and \eqref{E:majorminorarcsplitcase3} is carried out as before.

\subsection{Proof of Corollary \ref{C:pp-signchanges}} By \eqref{E:ppintermsofT}, we have
\begin{align} \label{eq:pp-diff}
    pp(a_1, b, n) - pp(a_2, b, n) = \frac{2}{b} \mathrm{Re}\left(\left(\zeta_b^{-a_1}-\zeta_b^{-a_2}\right) T_n(\zeta_b)\right) + \frac{1}{b}\sum_{2 \leq j \leq b-2} \left(\zeta_b^{-a_1j} - \zeta_{b}^{-a_2j}\right) T_n\left( \zeta_b^j \right).
\end{align}
Note that $0 < \frac{1}{b} \leq \frac{1}{3} < \theta_{12}$. By definition and Theorem \ref{T:PPasymp}, we find
\begin{align*}
   \frac{2}{b}\left(\zeta_b^{-a_1}-\zeta_b^{-a_2}\right) T_n(\zeta_b) & \sim \frac{2}{b}\left(\zeta_b^{-a_1}-\zeta_b^{-a_2}\right) \frac{\left(1-\zeta_b^a\right)^{\frac{1}{12}}\mathrm{Li}_3\left(\zeta_b^a\right)^{\frac{1}{6}}}{2^{\frac{1}{3}}\sqrt{3\pi}n^{\frac{2}{3}}}\exp\left(\frac{3}{2^{\frac{2}{3}}}\mathrm{Li}_3\left(\zeta_b^a\right)^{\frac{1}{3}}n^{\frac{2}{3}}\right) \\
   & = \frac{B_{a_1, a_2, b}e^{i\alpha_{a_1, a_2, b}}}{n^{\frac{2}{3}}} \exp\left(\frac{3}{2^{\frac{2}{3}}} \lambda_1 n^{\frac{2}{3}} \right) \exp\left(\frac{3}{2^{\frac{2}{3}}} \lambda_2 n^{\frac{2}{3}}i  \right) \\
      & = \frac{B_{a_1, a_2, b}}{n^{\frac{2}{3}}} \exp\left(\frac{3}{2^{\frac{2}{3}}} \lambda_1 n^{\frac{2}{3}} \right) \exp\left(\left( \alpha_{a_1, a_2, b} + \frac{3}{2^{\frac{2}{3}}} \lambda_2 n^{\frac{2}{3}}\right)i  \right).
   \end{align*}
   Hence, 
   \begin{align} \label{eq:complexasy}
    \frac{\frac{2}{b}\left(\zeta_b^{-a_1}-\zeta_b^{-a_2}\right) T_n(\zeta_b)}{B_{a_1, a_2, b} n^{-\frac{2}{3}} \exp\left(\frac{3}{2^{\frac{2}{3}}} \lambda_1 n^{\frac{2}{3}} \right)} = \exp\left(\left( \alpha_{a_1, a_2, b} + \frac{3}{2^{\frac{2}{3}}} \lambda_2 n^{\frac{2}{3}}\right)i  \right) + o(1).  
   \end{align}
When taking real parts in \eqref{eq:complexasy}, we find with $\R$-linearity 
\begin{align} \label{eq:realasy}
  \frac{\frac{2}{b} \mathrm{Re}\left(\left(\zeta_b^{-a_1}-\zeta_b^{-a_2}\right) T_n(\zeta_b)\right)}{B_{a_1, a_2, b} n^{-\frac{2}{3}} \exp\left(\frac{3}{2^{\frac{2}{3}}} \lambda_1 n^{\frac{2}{3}} \right)} = \cos\left( \alpha_{a_1, a_2, b} + \frac{3}{2^{\frac{2}{3}}} \lambda_2 n^{\frac{2}{3}}\right) + o(1).
\end{align}
Now, again with Theorem \ref{T:PPasymp} and Corollary \ref{C:TrilogDominant}, we have 
\begin{align} \label{eq:remainder}
    \frac{\frac{1}{b}\sum_{2 \leq j \leq b-2} \left(\zeta_b^{-a_1j} - \zeta_{b}^{-a_2j}\right) T_n\left( \zeta_b^j \right)}{B_{a_1, a_2, b} n^{-\frac{2}{3}} \exp\left(\frac{3}{2^{\frac{2}{3}}} \lambda_1 n^{\frac{2}{3}} \right)} = O\left( e^{-c_bn^{\frac23}}\right),
\end{align}
for some $c_b > 0$. The claim follows with \eqref{eq:pp-diff}, \eqref{eq:realasy} and \eqref{eq:remainder}.
\subsection{Proof of Theorem \ref{T:planeoverpartitions}}

In fact there is very little to change between the proofs of Theorems \ref{T:PPasymp} and \ref{T:planeoverpartitions}, as the bounds for error terms in Section 4.1, together with Proposition \ref{P:zeta3trilogdecreasing} to identify the major arc, can readily be applied.  With $N=\lfloor \delta n^{\frac{1}{3}} \rfloor$ as before, we write
\begin{align*}
    \mathcal{A}_n\left(\zeta_b^a\right)&=\sum_{\substack{\gcd(h,k)=1 \\ \frac{h}{k} \in \mathcal{F}_N}} \zeta_k^{nh} \int_{-\theta_{h,k}'}^{\theta_{h,k}''} \exp\left(\log \overline{PP}\left(\zeta_b^a;e^{-t_n+2\pi i \theta}\right) + nt_n-2\pi in \theta\right)d\theta \\
    &=\sum_{\substack{\gcd(h,k)=1 \\ \frac{h}{k} \in \mathcal{F}_N}} \zeta_k^{nh} \int_{-\theta_{h,k}'}^{\theta_{h,k}''} \exp\left(\frac{\zeta(3)-\mathrm{Li}_3\left(\zeta_b^{ka}\right)}{k^3(t_n-2\pi i \theta)^2} + nt_n-2\pi in \theta + \overline{E}_{h,k}(\zeta_{b}^a; t_n-2\pi i \theta)\right)d\theta,
\end{align*}
where 
$$
\overline{E}_{h,k}(z,t):=\log \overline{PP}\left(z;e^{-t}\right)-\frac{\zeta(3)-\mathrm{Li}_3\left(z^k\right)}{k^3t^2},
 \qquad 
t_n:=\frac{2^{\frac{1}{3}}\left(\zeta(3)-\mathrm{Li}_3\left(\zeta_b^a\right)\right)^{\frac{1}{3}}}{n^{\frac{1}{3}}}.
$$
Note that
$$
\Log\left(\overline{PP}(\zeta;q)\right)=\Log\left(PP(1;q)\right)-\Log\left(PP(\zeta;q)\right),
$$
so $$\overline{E}_{h,k}(z;t)=E_{0,1}(1;t)-E_{h,k}(z;t),$$ and we can show as in Lemma \ref{L:Ehk} that uniformly for $\theta\in [-\theta_{h,k}', \theta_{h,k}'']$ we have
\begin{equation}\label{E:Ehkbarbound}
\overline{E}_{h,k}(\zeta_b^a;t_n-2 \pi i \theta) = O(N^2).
\end{equation}
The only reason we cannot directly quote Lemma \ref{L:Ehk} is that our $t_n$ is a different constant times $n^{-\frac{1}{3}}.$  But this constant does not affect the proofs of Lemmas \ref{L:Ehk}, so we obtain \eqref{E:Ehkbarbound}.  The proof of Theorem \ref{T:planeoverpartitions} then follows Section 4.2 {\it mutatis mutandis}, this time using Proposition \ref{P:zeta3trilogdecreasing} to show that the major arc occurs at $\frac{h}{k}=\frac{0}{1}$ for all $\frac{a}{b}$.

\section{Outlook}\label{S:Outlook}

In general, one could ask for asymptotic formulae for the polynomials generated by products of the shape
$$
\prod_{n \geq 1} \frac{1}{(1-\zeta q^n)^{a_n}}, \qquad (a_n \geq 0),
$$
for $\zeta$ any root of unity.  This would generalize Meinardus' classical theorem \cite[Ch. 6]{Andrews} which is the case $\zeta=1$ and extend a theorem of Parry who proved asymptotic formulas when $|\zeta|<1.$  For Meinardus' theorem, one requires certain technical hypotheses on the sequence $\{a_n\}_{n \geq 1}$, and an extension of our method appears to require an excessive number of such hypotheses.  We thus leave the general case as an open problem.

Such an extension of Meinardus' theorem would  immediately yield information on the sign changes for many partition-theoretic objects.  One notable example comes from partitions into $\ell$-th powers, i.e.
$$
\sum_{n \geq 0} Q_n^{(\ell)}(\zeta)q^n=\prod_{n \geq 1} \frac{1}{1-\zeta q^{n^{\ell}}}.
$$
When $\zeta=-1$ the polynomials $Q_n^{(\ell)}(-1)$ count the difference of partitions into an even and odd number of $\ell$-th powers.  Ciolan \cite{Ciolan1, Ciolan2} overcame significant technical obstacles to prove asymptotic formulae for $Q^{(\ell)}_n(-1)$ and additionally gave a combinatorial proof of the following.

\begin{nonumbertheorem}[Theorem 2 of \cite{Ciolan2}] If $\ell \geq 2$, then
$$
Q^{(\ell)}_n(-1) > 0  \quad \text{(if $n$ is even),} \qquad \text{and} \qquad Q^{(\ell)}_n(-1) < 0  \quad \text{(if $n$ is odd).} 
$$
\end{nonumbertheorem}
\noindent It would thus be of great interest to see asymptotic formulae for $Q_n^{(\ell)}(\zeta)$ for arbitrary roots of unity $\zeta$.

In work of Corteel--Safelief--Vuleti\'{c}, there are also analogues for cylindric partitions of the statistics generated by $\mathcal{A}_n(\zeta)$, namely the polynomials generated by
$$
\prod_{\substack{n \geq 1 \\ j \in S}} \frac{1-\zeta q^{j+nT}}{1- q^{j+nT}},
$$
for some finite set $S$ (see \cite[Theorem 8]{CSV}).  We expect our techniques apply to these polynomials as well.

\end{document}